\documentclass[11pt]{amsart}
\usepackage[margin=1in]{geometry}
\usepackage{xcolor}
\usepackage{mathabx}
\usepackage[english]{babel}
\numberwithin{equation}{section}
\newtheorem{thm}{Theorem}[section]

\newtheorem{lem}[thm]{Lemma}

\newtheorem{defn}[thm]{Definition}

\def\XXint#1#2#3{{\setbox0=\hbox{$#1{#2#3}{\int}$ }
		\vcenter{\hbox{$#2#3$ }}\kern-.6\wd0}}

\newcommand{\io}{\int_\Omega}
\newcommand{\ot}{\Omega_t}
\newcommand{\oT}{\Omega_T}
\newcommand{\iot}{\int_{\Omega_t}}
\newcommand{\ioT}{\int_{\Omega_T}}
\newcommand{\ob}{\partial\Omega}
\newcommand{\pt}{\partial_t}
\newcommand{\sr}{\sqrt{\rho}}
\newcommand{\sm}{\sqrt{\mu}}
\newcommand{\sre}{\sqrt{\re}}
\newcommand{\sme}{\sqrt{\me}}
\newcommand{\ve}{\varepsilon}
\newcommand{\ra}{\rightarrow}
\newcommand{\vp}{\varphi}

\newcommand{\re}{\rho_\ve}
\newcommand{\me}{\mu_\ve}
\newcommand{\ret}{(\rho_\ve)^{(\sigma)}}
\newcommand{\met}{(\mu_\ve)^{(\sigma)}}

\newcommand{\ue}{u_\ve}

\newcommand{\wo}{w_1}
\newcommand{\wt}{w_2}
\newcommand{\wok}{w_1^{(k)}}
\newcommand{\wtk}{w_2^{(k)}}

\newcommand{\wob}{\overline{w}_1^{(j)}}

\newcommand{\rk}{\rho_k}
\newcommand{\mk}{\mu_k}
\newcommand{\rtj}{\tilde{\rho}_j}
\newcommand{\rbj}{\overline{\rho}_j}
\newcommand{\mtj}{\tilde{\mu}_j}
\newcommand{\mbj}{\overline{\mu}_{j}}
	
\begin{document}
	%	an approximation schemevia a  change of dependent variablesthat changes types
	\title[a partially parabolic cross-diffusion system]{Existence theorem for a partially parabolic cross-diffusion system  }
	\author{Xiangsheng Xu}\thanks
	%\addressNon-uniqueness of weak solutions to a doubly nonlinear fourth order elliptic equation Real variable Hardy spaces, The Green function,Only a partial regularity result holds if $N=3$. 
	{Department of Mathematics and Statistics, Mississippi State
		University, Mississippi State, MS 39762.
		{\it Email}: xxu@math.msstate.edu.}
%35B65, 35D35, 35M10, 35Q92, 35D30,  35A01, 35K67
	\begin{abstract} We study an initial boundary value problem for a cross-diffusion system in population dynamics. The mathematical challenge is due to the fact that the determinant of the coefficient matrix of the system changes signs. As a result, the system is only partially parabolic. We design an approximation scheme. The sequence of approximate solutions generated by our scheme converges and its limit  satisfies the original system in the parabolic region. It remains open if one can construct a vector-valued function that satisfies the system in both the parabolic region and the anti-parabolic one.
	\end{abstract}
	\maketitle

	\section{Introduction}\label{sec1}
	Let $\Omega$ be a bounded domain in $\mathbb{R}^N$ with Lipschitz boundary $\ob$. For each $T>0$ we consider the initial boundary value problem
	\begin{eqnarray}
		\rho_t&=&\Delta\rho+\textup{div}(\rho\nabla\mu)\ \ \mbox{in $\Omega_T\equiv\Omega\times(0,T)$},\label{eq1}\\
		\mu_t&=&\Delta\mu+\textup{div}(\mu\nabla\rho)\ \ \mbox{in $\Omega_T$},\label{eq2}\\
		\nabla\rho\cdot\mathbf{n}&=&\nabla\mu\cdot\mathbf{n}=0\ \ \mbox{on $\Sigma_T\equiv\ob\times(0,T)$},\label{eq3}\\
		\rho(x,0)&=&\rho_0(x),\ \mu(x,0)=\mu_0(x)\ \ \mbox{on $\Omega$},\label{eq4}
	\end{eqnarray}
where $\mathbf{n}$ is the unit outward to $\ob$. 
%Obviously, the initial data must satisfy the consistence condition
%\begin{equation}\label{icon}
%	\nabla\rho_0\cdot\mathbf{n}=\nabla\mu_0\cdot\mathbf{n}=0\ \ \mbox{on $\ob$.}
%\end{equation}

This problem can be proposed as a model for the motion of two competing populations \cite{BRYZ,DSY}. In this case, $\rho$ and $\mu$ are the two population densities. We refer the reader to \cite{AB,BRYZ} for detailed information on this.
%the physical relevance of this problem. 
Our interest here is the existence assertion for the problem.
The mathematical difficulty is due to the fact that the coefficient matrix
\begin{equation}\label{adef}
	A=\left(\begin{array}{ll}
		1&\rho\\
		\mu&1
	\end{array}\right)
\end{equation}
%$$$$
has an eigenvalue $1-\sqrt{\rho\mu}$, which changes sign. As observed in \cite{BRYZ,DSY}, the assumption $\sqrt{\rho_0(x)\mu_0(x)}<1$ can only ensure that $\sqrt{\rho\mu}\leq 1$ for a short time. Thus, our problem here seems to lie outside the scope of \cite{CJ}. To clarify this further, we
consider the general problem
\begin{eqnarray}
	\pt \rho-\mbox{div}\left(a_{11}(\rho,\mu)\nabla \rho+a_{12}(\rho,\mu)\nabla\mu\right)&=&0\ \ \mbox{in $\oT$},\label{ge1}\\
	\pt \mu-\mbox{div}\left(a_{21}(\rho,\mu)\nabla \rho+a_{22}(\rho,\mu)\nabla\mu\right)&=&0\ \ \mbox{in $\oT$},\label{ge2}\\
	\left(a_{11}(\rho,\mu)\nabla \rho+a_{12}(\rho,\mu)\nabla\mu\right)\cdot\mathbf{n}&=&0\ \ \mbox{on $\Sigma_T$},\label{ge3}\\
	\left(a_{21}(\rho,\mu)\nabla \rho+a_{22}(\rho,\mu)\nabla\mu\right)\cdot\mathbf{n}&=&0\ \ \mbox{on $\Sigma_T$},\label{ge4}\\
	(\rho(x,0), \mu(x,0))&=&(\rho_0(x), \mu_0(x))\ \ \mbox{on $\Omega$.}\label{ge5}
\end{eqnarray}
The idea in \cite{CJ} is to seek a convex function $\psi(\rho,\mu)$ so that $t\ra\io \psi(\rho,\mu)dx$ is a Lyapunov functional along the solution to \eqref{ge1}-\eqref{ge5}.  A simple calculation shows
% We calculate
\begin{eqnarray}
	\frac{d}{dt}\io\psi(\rho,\mu) dx&=&\io\left(\psi_\rho\pt \rho+\psi_\mu\pt \mu\right)\nonumber\\
	&=&-\io\left(a_{11}(\rho,\mu)\nabla\rho+a_{12}(\rho,\mu)\nabla\mu\right)\cdot\left(\psi_{\rho\rho}\nabla\rho+\psi_{\rho\mu}\nabla\mu\right)dx\nonumber\\
	&&-\io\left(a_{21}(\rho,\mu)\nabla\rho+a_{22}(\rho,\mu)\nabla\mu\right)\cdot\left(\psi_{\rho\mu}\nabla\rho+\psi_{\mu\mu}\nabla\mu\right)dx\nonumber\\
	&=&-\io\left[\left(a_{11}\psi_{\rho\rho}+a_{21}\psi_{\rho\mu}\right)|\nabla\rho|^2+\left(a_{12}\psi_{\rho\mu}+a_{22}\psi_{\mu\mu}\right)|\nabla\mu|^2\right]dx\nonumber\\
	&&-\io\left(a_{12}\psi_{\rho\rho}+(a_{11}+a_{22})\psi_{\rho\mu}+a_{21}\psi_{\mu\mu}\right)\nabla\rho\cdot\nabla\mu dx.\label{lyp10}
\end{eqnarray}
According to Proposition 1.3 in \cite{PX},
we must choose $\psi$ so that
\begin{eqnarray}
	a_{11}\psi_{\rho\rho}+a_{21}\psi_{\rho\mu}&\geq&0,\\
	a_{12}\psi_{\rho\mu}+a_{22}\psi_{\mu\mu}&\geq& 0,\\
	\left(a_{12}\psi_{\rho\rho}+(a_{11}+a_{22})\psi_{\rho\mu}+a_{21}\psi_{\mu\mu}\right)^2
	&\leq&4\left(a_{11}\psi_{\rho\rho}+a_{21}\psi_{\rho\mu}\right)\left(a_{12}\psi_{\rho\mu}+a_{22}\psi_{\mu\mu}\right).\label{pin}
\end{eqnarray}
Substitute \eqref{adef} into
\eqref{pin}  to get
\begin{eqnarray}
	%	\psi_{\rho\rho}+\mu\psi_{\mu\rho}&\geq&0,\label{lyp1}\\
	%	\rho\psi_{\mu\rho}+\psi_{\mu\mu}&\geq&0,\label{lyp2}\\
	(\rho\psi_{\rho\rho}+2\psi_{\mu\rho}+\mu\psi_{\mu\mu})^2&\leq&4(\psi_{\rho\rho}+\mu\psi_{\mu\rho})(	\rho\psi_{\mu\rho}+ \psi_{\mu\mu}).
\end{eqnarray}
Simplifying this equation yields
\begin{eqnarray}\label{lyp3}
	%	\lefteqn{
		4(1-\mu\rho)\psi_{\mu\rho}^2+\rho^2\psi_{\rho\rho}^2+\mu^2\psi_{\mu\mu}^2+2(\rho\mu-2)\psi_{\rho\rho}\psi_{\mu\mu}\leq0.
		%\nonumber\\	&=& 4AD-4D(\psi_{\rho\rho}+\mu\psi_{\mu\rho})-4A(\rho\psi_{\mu\rho}+ {\mu\mu}).
	\end{eqnarray}
	Denote by $\nabla^2\psi$ the Hessian matrix of $\psi$, i.e.,
	$$\nabla^2\psi=\left(\begin{array}{ll}
		\psi_{\rho\rho}&\psi_{\mu\rho}\\
		\psi_{\rho\mu}&\psi_{\mu\mu}
	\end{array}\right).$$
	We can write \eqref{lyp3} as
	\begin{eqnarray}
		4(1-\mu\rho)\textup{det}\nabla^2\psi\geq\left(\rho\psi_{\rho\rho}-\mu\psi_{\mu\mu}\right)^2.\label{lyp2}
	\end{eqnarray}
	This asserts that $\psi$ can only be convex on the set $\{\rho\mu\leq 1\}$. That is, it is not possible for us to find a convex function $\psi(\rho,\mu)$ on $[0,\infty)\times [0,\infty)$ so that $t\ra\io \psi(\rho,\mu)dx$ is decreasing.
	
	 In \cite{DSY}, the authors formulated the problem as a gradient flow of a functional in a metric space. Unfortunately, the functional is not convex, and the authors must replace it with its convex envelope in order to obtain an existence theorem via the gradient flow theory. Our objective here is to offer a new approach.

%The existence of a
%A weak solution is often constructed as the limit of a sequence of approximate solutions. Our first step is to design a suitable approximation scheme. %That is to say, it must satisfy the following two conditions. 
%Obviously, we  must be able to show the existence assertion for the approximate problems. Furthermore, they must retain the a priori estimates satisfied by the original problem. It is not clear to us how we can design such an approximation scheme for \eqref{eq1}-\eqref{eq4} directly. % That is, we come up with a family of approximate problems.
% indexed with a parameter $\ve$. 
%We can prove the existence assertion for them and their solutions satisfy the same a priori estimates as the one to the original system.

To motivate our approach, we proceed to derive a priori estimates \cite{BRYZ}. For this purpose, we assume that $\rho, \mu$ are positive and sufficiently regular. Write the right hand side of \eqref{eq1} as
\begin{eqnarray}
	\Delta\rho+\textup{div}(\rho\nabla\mu)&=&\textup{div}\left[\rho\nabla(\ln\rho+\mu)\right].
\end{eqnarray}
With this in mind, we use $\ln\rho+\mu$ as a test function in \eqref{eq1} to deduce
\begin{eqnarray}
	\io(\ln\rho+\mu)\rho_t dx&=&-\io\rho\left|\nabla(\ln\rho+\mu)\right|^2dx=-4\io\left|\nabla\sr+\sqrt{\rho\mu}\nabla\sm\right|^2dx.\label{cd1}
\end{eqnarray}
Similarly, we can deduce from \eqref{eq2} that
\begin{eqnarray}
	\io(\ln\mu+\rho)\mu_t dx&=&-4\io|\nabla\sm+\sqrt{\rho\mu}\nabla\sr|^2dx.
\end{eqnarray}
Adding this to \eqref{cd1} yields
\begin{eqnarray}
\lefteqn{	\io(\rho\ln\rho-\rho+\mu\ln\mu-\mu+\rho\mu)dx}\nonumber\\
&&+4\iot\left(|\nabla\sr+\sqrt{\rho\mu}\nabla\sm|^2+|\nabla\sm+\sqrt{\rho\mu}\nabla\sr|^2\right)dxd\tau\nonumber\\
&=&\io(\rho_0\ln\rho_0-\rho_0+\mu_0\ln\mu_0-\mu_0+\rho_0\mu_0)dx,\label{cd2}
\end{eqnarray}
where 
$$\ot=\Omega\times(0,t).$$
Recall the elementary inequality
\begin{equation}\label{eli}
	2\left(|x|^2+|y|^2\right)\geq |x\pm y|^2\ \ \mbox{for $x,y\in \mathbb{R}^N$,}
\end{equation}
%$$$$
from whence follows
\begin{eqnarray*}
	2\left(|\nabla\sr+\sqrt{\rho\mu}\nabla\sm|^2+|\nabla\sm+\sqrt{\rho\mu}\nabla\sr|^2\right)&\geq &(1+\sqrt{\rho\mu})^2|\nabla\sr+\nabla\sm|^2,\label{ues}\\
	2\left(|\nabla\sr+\sqrt{\rho\mu}\nabla\sm|^2+|\nabla\sm+\sqrt{\rho\mu}\nabla\sr|^2\right)&\geq &(1-\sqrt{\rho\mu})^2|\nabla\sr-\nabla\sm|^2.\label{ves}
\end{eqnarray*}
Combine the preceding results to deduce
\begin{eqnarray}
	\lefteqn{\sup_{0\leq t\leq T}	\io(\rho\ln\rho-\rho+\mu\ln\mu-\mu+\rho\mu)dx}\nonumber\\
	&&+2\ioT(1+\sqrt{\rho\mu})^2|\nabla\sr+\nabla\sm|^2dxdt+2\ioT(1-\sqrt{\rho\mu})^2|\nabla\sr-\nabla\sm|^2dxdt\nonumber\\
	&\leq& 2\io(\rho_0\ln\rho_0-\rho_0+\mu_0\ln\mu_0-\mu_0+\rho_0\mu_0)dx.\label{me}
\end{eqnarray}
%We easily verify that
%\begin{eqnarray}
%	\lefteqn{(1+\sqrt{\rho\mu})^2|\nabla\sr+\nabla\sm|^2+(1-\sqrt{\rho\mu})^2|\nabla\sr-\nabla\sm|^2}\nonumber\\
%	&=&2(1+\rho\mu)|\nabla\sr|^2+8\sqrt{\rho\mu}\nabla\sr\cdot\nabla\sm+2(1+\rho\mu)|\nabla\sm|^2\nonumber\\
%	&\geq&
%\end{eqnarray}
%The preceding calculations seem to suggest that t
The so-called entropy variables as suggested in \cite{CDJ} are
$$u_1=\ln\rho+\mu,\ \ u_2=\ln\mu+\rho.$$
Unfortunately, the Jacobian for the transformation is $\frac{1-\rho\mu}{\rho\mu}$, which changes sign. 
%This seems to suggest that we should 
%consider the following two functions %
%introduce 
%It turns out that the right approach is to introduce the change of dependent variables
To circumvent this problem, we consider
\begin{equation}\label{ho0}%1-\sqrt{\rho\mu}
		u=\sr+\sm,\ \ 	v=\sr-\sm.
\end{equation}
To derive an equation satisfied by $u$, we first write system \eqref{eq1}-\eqref{eq2} in terms of $\sqrt{\rho}, \sqrt{\mu}$. Starting with \eqref{eq1}, we obtain
\begin{eqnarray*}
	2\sqrt{\rho}\pt\sr&=&\mbox{div}\left[2\sr\left(\nabla\sr+\sqrt{\rho\mu}\nabla\sm\right)\right]\nonumber\\
	&=&2\sr\mbox{div}\left(\nabla\sr+\sqrt{\rho\mu}\nabla\sm\right)+2\nabla\sr\cdot\left(\nabla\sr+\sqrt{\rho\mu}\nabla\sm\right),
\end{eqnarray*}
from whence follows
\begin{eqnarray*}\label{sre}%
	\pt\sr&=&\mbox{div}\left(\nabla\sr+\sqrt{\rho\mu}\nabla\sm\right)+
	%F(\sr,\sm),%+
	\frac{1}{\sr}\left(|\nabla\sr|^2+\sqrt{\rho\mu}\nabla\sm\cdot\nabla\sr\right).
\end{eqnarray*}%where
%\begin{eqnarray}
%F(\xi,\eta)&=&\frac{1}{\xi}\left(|\nabla\xi|^2+\xi\eta\nabla\xi\cdot\nabla\eta\right).%+\frac{1}{\eta}\left(|\nabla\eta|^2+\xi\eta\nabla\xi\cdot\nabla\eta\right).%\nonumbe%r\\
%&=&4|\nabla\sqrt{\xi}|^2+(\eta+\xi)\nabla\xi\cdot\nabla\eta+4|\nabla\sqrt{\eta}|^2.F(\sr,\sm),
%\frac{1}{\xi}|\nabla\xi|^2+(\sqrt{\rho}+\sqrt{\mu})\nabla\xi\cdot\nabla\eta+\frac{1}{\eta}|\nabla\eta|^2.\ \ \mbox{for $\xi\geq0,\eta\geq 0$}
%\end{eqnarray}
Similarly,  %by setting 
%$$	G(\xi,\eta)=\frac{1}{\eta}\left(|\nabla\eta|^2+\xi\eta\nabla\xi\cdot\nabla\eta\right),$$
we can obtain from \eqref{eq2} that 
\begin{equation*}\label{ho1}
	\pt\sm=\mbox{div}\left(\nabla\sm+\sqrt{\rho\mu}\nabla\sr\right)+\frac{1}{\sm}\left(|\nabla\sm|^2+\sqrt{\rho\mu}\nabla\sm\cdot\nabla\sr\right).
%	+G(\sr,\sm).
\end{equation*}
Add and subtract the preceding two equations successively to deduce
\begin{eqnarray}
	\pt u&=&\mbox{div}\left[(1+\sqrt{\rho\mu})\nabla u\right]+%+F(\sr,\sm)+G(\sr,\sm)%
	\frac{1}{\sr}\left(|\nabla\sr|^2+\sqrt{\rho\mu}\nabla\sm\cdot\nabla\sr\right)\nonumber\\
	&&+\frac{1}{\sm}\left(|\nabla\sm|^2+\sqrt{\rho\mu}\nabla\sm\cdot\nabla\sr\right)
	%+F(\sr,\sm)+G(\sr,\sm)%\frac{1}{\sr}\left(|\nabla\sr|^2+\sqrt{\rho\mu}\nabla\sm\cdot\nabla\sr\right)\nonumber\\
%	&&+\frac{1}{\sm}\left(|\nabla\sm|^2+\sqrt{\rho\mu}\nabla\sm\cdot\nabla\sr\right)
,\label{ho4}\\
%	\nonumber\\&=&\mbox{div}\left[(1+\sqrt{\rho\mu})\nabla u\right]+\frac{1}{4\rho^{\frac{3}{2}}}|\nabla\rho|^2+\frac{\sqrt{\rho}+\sqrt{\mu}}{4\sqrt{\rho\mu}}\nabla\rho\cdot\nabla\mu+\frac{1}{4\mu^{\frac{3}{2}}}|\nabla\mu|^2.
\pt v&=&\mbox{div}\left[(1-\sqrt{\rho\mu})\nabla v\right]+	+	\frac{1}{\sr}\left(|\nabla\sr|^2+\sqrt{\rho\mu}\nabla\sm\cdot\nabla\sr\right)\nonumber\\
&&-\frac{1}{\sm}\left(|\nabla\sm|^2+\sqrt{\rho\mu}\nabla\sm\cdot\nabla\sr\right).\label{ho3}
%F(\sr,\sm)-G(\sr,\sm).%\sqrt{\rho\mu}\nabla\sm\cdot\nabla\sr\right)\nonumber\\
	%&&-\frac{1}{\sm}\left(|\nabla\sm|^2+\sqrt{\rho\mu}\nabla\sm\cdot\nabla\sr\right)	.
\end{eqnarray}
The equation for $v$ is parabolic only on the set $\{\sqrt{\rho\mu}< 1\}$, while outside the set the diffusion coefficient in the equation becomes negative and the resulting problem is ill-posed. A reasonable solution would be to replace the  coefficient $1-\sqrt{\rho\mu}$ in \eqref{ho3} with $\left(1-\sqrt{\rho\mu}\right)^+$.
%the system with $\theta_1(\sqrt{\rho\mu})$, where
Introduce the function
\begin{equation}\label{tld}
	\theta_\ell(s)=\left\{\begin{array}{ll}
	\ell&\mbox{if $s\geq \ell$,}\\
		s&\mbox{if $0<s<\ell$,}\\
		0&\mbox{if $s\leq 0,\ \ \ell>0$.}
	\end{array}\right.
\end{equation}% $\theta$ can be any twice differentiable function on $\mathbb{R}$ with the properties
%\begin{eqnarray}
%\theta_1(s)=0\ \ \mbox{on $(-\infty,0)$, $0\leq \theta\leq 1$, and $\theta_1(s)=1$ on $[1,\infty)$,}\label{tpro}	
%\end{eqnarray}
%and this is what we will do. 
Subsequently, 
$$(1-\sqrt{\rho\mu})^+=1-\theta_1(\sqrt{\rho\mu}).$$
We are led to  the system
\begin{eqnarray}
	\pt u&=&\mbox{div}\left[(1+\theta_1(\sqrt{\rho\mu}))\nabla u\right]+%+F(\sr,\sm)+G(\sr,\sm)%
	\frac{1}{\sr}\left(|\nabla\sr|^2+\theta_1(\sqrt{\rho\mu})\nabla\sm\cdot\nabla\sr\right)\nonumber\\
		&&+\frac{1}{\sm}\left(|\nabla\sm|^2+\theta_1(\sqrt{\rho\mu})\nabla\sm\cdot\nabla\sr\right)
	,\label{ho2}\\
	%	\nonumber\\&=&\mbox{div}\left[(1+\sqrt{\rho\mu})\nabla u\right]+\frac{1}{4\rho^{\frac{3}{2}}}|\nabla\rho|^2+\frac{\sqrt{\rho}+\sqrt{\mu}}{4\sqrt{\rho\mu}}\nabla\rho\cdot\nabla\mu+\frac{1}{4\mu^{\frac{3}{2}}}|\nabla\mu|^2.
	\pt v&=&\mbox{div}\left[(1-\theta_1(\sqrt{\rho\mu}))\nabla v\right]%+F(\sr,\sm)-G(\sr,\sm)%
	+	\frac{1}{\sr}\left(|\nabla\sr|^2+\theta_1(\sqrt{\rho\mu})\nabla\sm\cdot\nabla\sr\right)\nonumber\\
	&&-\frac{1}{\sm}\left(|\nabla\sm|^2+\theta_1(\sqrt{\rho\mu})\nabla\sm\cdot\nabla\sr\right).
\label{hop10}
\end{eqnarray}
Here we have replaced $\sqrt{\rho\mu}$ with $\theta_1(\sqrt{\rho\mu})$ throughout system \eqref{ho4}-\eqref{ho3}. This is done so that we can still write the right-hand side terms in \eqref{ho2} and \eqref{hop10} as divergences. Indeed,
 %, our new problem agrees with the old one only on the set $\{\sqrt{\rho\mu}< 1\}$.
 % For convenience, use $\xi$ for $\sr$ and $\eta$ for $\sm$, respectively.
% Set
% \begin{equation}\label{adef}
 %	a_\ve(\xi,\eta)=1-(1-\xi\eta)^+-\ve.
% \end{equation} 
%It is much easier for us 
%We are ready to design an approximation to system \eqref{ho2}-\eqref{hop10}. To this end, 	
%we fix $\ve\in(0,1]$. Later, we will take
%\begin{equation}\label{edf}
%	\ve=\frac{1}{k},\ \  k=1,2,\cdots.
%\end{equation}\equiv F(\sr,\sm)\equiv G(\sr,\sm)
%We can form our approximate problems as follows:
%Now system \eqref{sre}-\eqref{ho1} become
add and subtract the two equations \eqref{ho2} and \eqref{hop10}, respectively, and keep \eqref{ho0} in mind to get
 \begin{eqnarray}
 	\lefteqn{	\pt\sr-\mbox{div}\left[\nabla\sr+\theta_1(\sqrt{\rho\mu})\nabla\sm\right]}\nonumber\\
 	&=&\frac{1}{\sr}\left(|\nabla\sr|^2+\theta_1(\sqrt{\rho\mu})\nabla\sr\cdot\nabla\sm\right)\ \ \mbox{in $\oT$},\label{ax1}\\
 	\lefteqn{	\pt\sm-\mbox{div}\left[\nabla\sm+\theta_1(\sqrt{\rho\mu})\nabla\sr\right]}\nonumber\\
 	&=&\frac{1}{\sm}\left(|\nabla\sm|^2+\theta_1(\sqrt{\rho\mu})\nabla\sr\cdot\nabla\sm\right)\ \ \mbox{in $\oT$}.\label{ae1}%\\
 %	\nabla\sr\cdot\mathbf{n}&=&\nabla\sm\cdot\mathbf{n}=0\ \ \mbox{on $\Sigma_T$,}\label{xeb}\\
 %	\sr(x,0)&=&\sqrt{\rho_0(x)}\ \ \mbox{on $\Omega$,}\ \ \sm(x,0)=\sqrt{\mu_0(x)}.\label{xei}
 \end{eqnarray}
Multiply through \eqref{ax1} by $2\sr$ to derive 
\begin{eqnarray}
	\pt\rho&=&\mbox{div}\left[\nabla\rho+\frac{\sqrt{\rho}\theta_1(\sqrt{\rho\mu})}{\sqrt{\mu}}\nabla\mu\right]\ \ \mbox{in $\oT$.}\label{pl1}
\end{eqnarray}
Similarly, we can deduce from \eqref{ae1} that
\begin{eqnarray}
	\pt\mu&=&\mbox{div}\left[\nabla\mu+\frac{\sqrt{\mu}\theta_1(\sqrt{\rho\mu})}{\sqrt{\rho}}\nabla\rho\right]	\ \ \mbox{in $\oT$.}\label{pl2}
\end{eqnarray}
By \eqref{tld}, we have
\begin{equation*}\label{pl0}
	\frac{\sqrt{\rho}\theta_1(\sqrt{\rho\mu})}{\sqrt{\mu}}=\frac{\theta_1(\sqrt{\rho\mu})}{\sqrt{\rho\mu}}\rho\leq\rho,\ \ \frac{\sqrt{\mu}\theta_1(\sqrt{\rho\mu})}{\sqrt{\rho}}=\frac{\theta_1(\sqrt{\rho\mu})}{\sqrt{\rho\mu}}\mu\leq \mu. 
\end{equation*}
However, the equality holds on the set $\{\rho\mu\leq 1\}$.
As we mentioned earlier, a modification to \eqref{eq1}-\eqref{eq4} is also carried out in \cite{DSY} under a different context.
%Our modification here is further justified by and singularity, while singularity occurs at either $\rho=0$ or $\mu=0$
This seems unavoidable in view of the observation  in \cite{DSY} that even if $\rho_0\mu_0<1$ we can not expect $\rho\mu\leq 1$ for all $t>0$. In spite of our modification here, the resulting system \eqref{pl1}-\eqref{pl2} is still a cross-diffusion one with degeneracy. Degeneracy is due to the fact that
 the determinant of the coefficient matrix for the system is only non-negative. %But it is still not convenient for us to take the limit as $\ve\ra 0$ in the system due to the fractions. 
As a result, we are not able to obtain any estimates on $\nabla\rho, \nabla\mu$ over the whole domain $\oT$. To get around this issue, 
%To make it more convenient for us to pass to the limit, 
we substitute $u=\sqrt{\rho}+\sqrt{\mu}$ into \eqref{pl1} and \eqref{pl2} to derive
\begin{eqnarray}
	\pt\rho&=&2\mbox{div}\left[\sr\left(1-\theta_1(\sqrt{\rho\mu})\right)\nabla\sr+\sqrt{\rho}\theta_1(\sqrt{\rho\mu})\nabla u\right]\ \ \mbox{in $\oT$,}\label{pl5}\\
	\pt\mu&=&2\mbox{div}\left[\sm(1-\theta_1(\sqrt{\rho\mu}))\nabla\sm+\sqrt{\mu}\theta_1(\sqrt{\rho\mu})\nabla u\right]	\ \ \mbox{in $\oT$.}\label{pl6}
\end{eqnarray}
This action makes sense due to our estimate \eqref{me}.

%We need to introduce the following function space as in \cite{X4}. Let $f\in L^2(\oT)$ be such that
%\begin{equation}
%	a(f)\in L^1(0,T;W^{1,1}(\Omega))\ \ \mbox{for each $a\in C^2_0(-\infty,1)$ with }.
%\end{equation}
%Then define
%$$X_f=\{\}$$
\begin{defn}\label{def} A quadruplet $(\rho,\mu, \mathbf{f}_1, \mathbf{f}_2)$ is said to be a weak solution to system \eqref{pl5}-\eqref{pl6} coupled with \eqref{eq3} and \eqref{eq4} %\eqref{eq1}-\eqref{eq4} 
	if the following hold:
	\begin{enumerate}
		\item[(\textup{D1})] $\rho\geq 0,\ \ \mu\geq 0,\ \ \rho\ln\rho, \ \mu\ln\mu\in L^\infty(0,T; L^1(\Omega)),\ u\equiv\sr+\sm\in  L^2(0,T;W^{1,2}(\Omega)),\ \mathbf{f}_1, \mathbf{f}_2\in (L^1(\oT))^N$ with
		 \begin{eqnarray*}
			\mathbf{f}_1&=&\sr\sqrt{\rho\mu}\nabla u\  \ \mbox{on $\{\sqrt{\rho\mu}<1\}$},\\
			\mathbf{f}_2&=&\sm\sqrt{\rho\mu}\nabla u\  \ \mbox{on $\{\sqrt{\rho\mu}<1\}$};
		\end{eqnarray*}
			\item[(\textup{D2})] $a(\sqrt{\rho\mu}),\ a(\sqrt{\rho\mu}) b (\sr),\  a(\sqrt{\rho\mu}) b (\sm)
			\in L^1(0,T; W^{1,1}(\Omega))$ for each $a\in W^{2,\infty}_0(-\infty,1)$ and each $b\in C^1_{\textup{c}}(\mathbb{R}), \  \sr(1-\theta_1(\sqrt{\rho\mu}))\nabla\sr,\  \sm(1-\theta_1(\sqrt{\rho\mu}))\nabla\sm\in \left(L^1(\oT)\right)^N,$
				where %for each $\ell>0$ we have(1-\theta_1(\sqrt{\rho\mu}))
			\begin{eqnarray*}
				\nabla\sr&=&\nabla\left[a(\sqrt{\rho\mu}) b (\sr)\right]- b (\sr)\nabla a(\sqrt{\rho\mu})\ \ \mbox{in $\{\sqrt{\rho\mu}<1-\delta\}\cap\{\sr<\frac{1}{1-\delta}\}$},\\
					\nabla\sm&=&\nabla\left[a(\sqrt{\rho\mu}) b (\sm)\right]- b (\sm)\nabla a(\sqrt{\rho\mu})\ \ \mbox{in $\{\sqrt{\rho\mu}<1-\delta\}\cap\{\sr<\frac{1}{1-\delta}\}$} 
			\end{eqnarray*}
		for each $ \delta\in(0,1)$,  each  $a\in W^{2,\infty}_0(-\infty,1)$ such that $a(s)=1$ on $[0,1-\delta]$, and each $b\in C^1_{\textup{c}}(\mathbb{R})$ with $b(s)=s$ on $[0,\frac{1}{1-\delta}]$ , while $\sr(1-\theta_1(\sqrt{\rho\mu}))\nabla\sr$ and $\sm(1-\theta_1(\sqrt{\rho\mu}))\nabla\sm$ are understood to be $\mathbf{0}$ on the set $\{\sqrt{\rho\mu}\geq 1\}$ ;
				\item[(\textup{D3})] There hold
				\begin{eqnarray*}
					\ioT\rho\pt\vp_1dxdt&=&2\ioT\left[\sr(1-\theta_1(\sqrt{\rho\mu}))\nabla\sr+\mathbf{f}_1\right]\cdot\nabla\vp_1 dxdt-\io\rho_0(x)\vp_1(x,0)dx,\\
					\ioT\mu\pt\vp_2dxdt&=&2\ioT\left[\sm(1-\theta_1(\sqrt{\rho\mu}))\nabla\sm+\mathbf{f}_2\right]\cdot\nabla\vp_2 dxdt-\io\mu_0(x)\vp_1(x,0)dx
				\end{eqnarray*}
			for all $(\vp_1, \vp_2)\in \left( C^1(\overline{\oT})\right)^2$ with $(\vp_1(x,T), \vp_2(x,T))=(0,0)$.
			%, a(\sqrt{\rho\mu})\sr, a(\sqrt{\rho\mu})\sm
			\end{enumerate}
	\end{defn}
It is not surprising that we are only able to characterize $\mathbf{f}_1,\ \mathbf{f}_2$ on the set $\{\sqrt{\rho\mu}<1\}$ because system \eqref{pl5}-\eqref{pl6} is degenerate on the set $\{\sqrt{\rho\mu}\geq 1\}$. Our definition allows the possibility that $\nabla\sr$ is a pure distribution in the whole domain $\oT$. Condition (D2) means that $\nabla\sr$ can be identified with a (vector-valued) function on the set $\{\sqrt{\rho\mu}<1\}$. We refer the reader to \cite{X4,X5} for a more detailed discussion on this. The same remark also applies to $\nabla\sm$.
\begin{thm}[Main Theorem] Assume:
	\begin{enumerate}
		%, \ \rho_0(x)\mu_0(x)
		\item[\textup{(H1)}] $\ob$ is Lipschitz;
		\item[\textup{(H2)}]$\rho_0(x)\geq 0,\ \ \mu_0(x)\geq 0,\ \ \rho_0(x)\ln\rho_0(x), \ \mu_0(x)\ln\mu_0(x)\in L^1(\Omega)$.
	\end{enumerate}
	Then there is a weak solution $(\rho,\mu)$ to system \eqref{pl5}-\eqref{pl6} coupled with \eqref{eq3} and \eqref{eq4} in the sense of Definition \ref{def}.
\end{thm}

Our plan is as follows: First, we show that for each $\ve\in(0,1)$ there is a ``classical'' weak solution to the problem
\begin{eqnarray}
		\pt\re&=&\mbox{div}\left[(1+\ve)\nabla\re+\frac{\theta_{\frac{1}{\ve}}(\sqrt{\re})\theta_1(\sqrt{\re\me})}{\sqrt{\me}}\nabla\me\right]\nonumber\\
	&=&\mbox{div}\left[2\sre(1+\ve)\nabla\sre+2\theta_{\frac{1}{\ve}}(\sqrt{\re})\theta_1(\sqrt{\re\me})\nabla\sme\right]
	\ \ \mbox{in $\oT$,}\label{appx1}\\
		\pt\me&=&\mbox{div}\left[(1+\ve)\nabla\me+\frac{\theta_{\frac{1}{\ve}}(\sqrt{\me})\theta_1(\sqrt{\re\me})}{\sqrt{\re}}\nabla\re\right]	\nonumber\\
		&=&\mbox{div}\left[2\sme\left((1+\ve)\nabla\sme+2\theta_{\frac{1}{\ve}}(\sqrt{\me})\theta_1(\sqrt{\re\me})\nabla\sre\right)\right]\ \ \mbox{in $\oT$,}\label{appx2}\\
		%	&=&\mbox{div}\left[2(1+\ve)\sre\nabla\sre+2\sqrt{\re}\theta_1(\sqrt{\re\me})\nabla\sme\right]
		\nabla\re\cdot\mathbf{n}&=&\nabla\me\cdot\mathbf{n}=0\ \ \mbox{on $\Sigma_T$},\label{appx3}\\
		\re(x,0)&=&\theta_{\frac{1}{\ve}}(\rho_0(x))+\ve,\ \ 	\me(x,0)=\theta_{\frac{1}{\ve}}(\mu_0(x))+\ve\ \ \mbox{on $\Omega$.}\label{appx4}
\end{eqnarray} 
%where $\gamma_\ve$ is a function in $C^\infty_0(\mathbb{R})$ with the property
%\begin{equation}
%	\gamma_\ve(s)=s\ \ \mbox{for $s\in[0,\frac{1}{\ve}-\ve]$ and $0\leq\gamma_\ve$.}
%\end{equation}
This will be done in Section \ref{sec2}. Obviously, the introduction of $\ve$ here is made so that the determinant of the coefficient matrix for the system
\begin{equation}\label{det1}
	\mbox{det}\left(\begin{array}{ll}
		1+\ve&\frac{\theta_{\frac{1}{\ve}}(\sqrt{\re})\theta_1(\sqrt{\re\me})}{\sqrt{\me}}\\
		  \frac{\theta_{\frac{1}{\ve}}(\sqrt{\me})\theta_1(\sqrt{\re\me})}{\sqrt{\re}}         &1+\ve
		\end{array}	\right)=(1+\ve)^2-\frac{\theta_{\frac{1}{\ve}}(\sqrt{\re})\theta_{\frac{1}{\ve}}(\sqrt{\me})}{\sqrt{\re\me}}\theta_1^2(\sqrt{\re\me})>0.
\end{equation}
We will explain why we must use the function $\theta_{\frac{1}{\ve}}(s)$ in our approximation scheme later. Substitute
\begin{equation}\label{ue}
	u_\ve=\sqrt{\re}+\sqrt{\me}
\end{equation}
into \eqref{appx1}-\eqref{appx2} and take $\ve\ra 0$ in the resulting system to arrive at our main theorem. Detailed analysis of this process will be provided in Section \ref{sec3}. We must point out that our situation here is different from the weak stability result in \cite{BRYZ}, where the authors have assumed that the sequence $\{(\sqrt{\re},\sqrt{\me})\}$ is bounded in $\left(L^2(0,T;W^{1,2}(\Omega))\right)^2$. Unfortunately, as pointed out in \cite{BRYZ}, this assumption is not known to hold.
%To this end, we observe
%\begin{equation}\label{ec}
%	\sqrt{\rho\mu}=1-(1-\sqrt{\rho\mu})=1-(1-\sqrt{\rho\mu})^+\ \ \mbox{on $\{\sqrt{\rho\mu}\leq 1\}$.}
%\end{equation}\eqref{ue1}-\eqref{ve1}
%\section{}\label{sec2}
\section{Existence assertion for approximate problems}\label{sec2}
In this section, 
we first state a preparatory lemma. Then we will construct a solution to \eqref{appx1}-\eqref{appx4}.

\begin{lem}[Lions-Aubin]\label{la}
	Let $X_0, X$ and $X_1$ be three Banach spaces with $X_0 \subseteq X \subseteq X_1$. Suppose that $X_0$ is compactly embedded in $X$ and that $X$ is continuously embedded in $X_1$. For $1 \leq p, q \leq \infty$, let
	\begin{equation*}
		W=\{u\in L^{p}([0,T];X_{0}): \partial_t u\in L^{q}([0,T];X_{1})\}.
	\end{equation*}
	%	{\displaystyle}
	Then:
	\begin{enumerate}
		\item[\textup{(i)}] If $p  < \infty$, then the embedding of W into $L^p([0, T]; X)$ is compact.
		\item[\textup{(ii)}] If $p  = \infty$ and $q  >  1$, then the embedding of W into $C([0, T]; X)$ is compact. 
	\end{enumerate}
\end{lem}
We refer the reader to \cite{S} for the proof of this lemma. More recent results in this direction can be found in \cite{CJL,DJ}. Of particular interest to us is Theorem 1 in \cite{DJ}.

The establishment of an existence assertion to \eqref{appx1}-\eqref{appx4} will largely follow the ideas in \cite{CDJ,CJ,J}. For simplicity, we will drop the subscript $\ve$ in \eqref{appx1}-\eqref{appx4}. That is, we shall write $(\rho,\mu)$ for $(\re,\me)$.
%Let $\rho,\mu$ satisfy . Since
% the determinant of the coefficient matrix in the system  is non-negative, we are motivated 
We are motivated by  \eqref{det1}  to seek a convex function $\psi(\rho,\mu)$ so that $\frac{d}{dt}\io\psi(\rho,\mu)dx\leq 0$ whenever $(\rho,\mu)$ is a classical solution to system \eqref{appx1}-\eqref{appx2} with $\rho>0,\ \mu>0$.
%It follows from 
According to \eqref{pin}, it is enough for us to pick $\psi(\rho,\mu)$ with
\begin{eqnarray}
	\lefteqn{\left(\frac{\theta_{\frac{1}{\ve}}(\sqrt{\rho})\theta_1(\sqrt{\rho\mu})}{\sqrt{\mu}}\psi_{\rho\rho}+2(1+\ve)\psi_{\rho\mu}+\frac{\theta_{\frac{1}{\ve}}(\sqrt{\mu})\theta_1(\sqrt{\rho\mu})}{\sqrt{\rho}}\psi_{\mu\mu}\right)^2}\nonumber\\
	&\leq&4\left((1+\ve)\psi_{\rho\rho}+\frac{\theta_{\frac{1}{\ve}}(\sqrt{\mu})\theta_1(\sqrt{\rho\mu})}{\sqrt{\rho}}\psi_{\rho\mu}\right)\left(\frac{\theta_{\frac{1}{\ve}}(\sqrt{\rho})\theta_1(\sqrt{\rho\mu})}{\sqrt{\mu}}\psi_{\rho\mu}+(1+\ve)\psi_{\mu\mu}\right).
\end{eqnarray}
By the proof of \eqref{lyp2}, the above inequality is equivalent to
\begin{eqnarray}
\lefteqn{	\left(\frac{\theta_{\frac{1}{\ve}}(\sqrt{\rho})\theta_1(\sqrt{\rho\mu})}{\sqrt{\mu}}\psi_{\rho\rho}-\frac{\theta_{\frac{1}{\ve}}(\sqrt{\mu})\theta_1(\sqrt{\rho\mu})}{\sqrt{\rho}}\psi_{\mu\mu}\right)^2}\nonumber\\
&\leq& 4\left((1+\ve)^2-\frac{\theta_{\frac{1}{\ve}}(\sqrt{\rho})\theta_{\frac{1}{\ve}}(\sqrt{\mu})}{\sqrt{\rho\mu}}\theta_1^2(\sqrt{\rho\mu})\right)\mbox{det}(\nabla^2\psi).\label{ms1}
\end{eqnarray}
%Any convex solutions of the equations
%\begin{eqnarray}
%	\rho\psi_{\rho\rho}-\mu\psi_{\mu\mu}=0.
%\end{eqnarray}
%will serve our purpose.
We easily verify that
$$\psi(\rho,\mu)=\rho\ln\rho+\mu\ln\mu$$
is a solution of \eqref{ms1}. Set
%where
\begin{equation}\label{sdef}
S_\ve=\frac{\theta_{\frac{1}{\ve}}(\sr)}{\sr}+\frac{\theta_{\frac{1}{\ve}}(\sm)}{\sm}\leq 2.	
\end{equation}
%$$$$
It follows from \eqref{lyp10} that
\begin{eqnarray}
	\lefteqn{	\frac{d}{dt}\io(\rho\ln\rho+\mu\ln\mu)dx}\nonumber\\
	&=&-\io\left(\frac{1+\ve}{\rho}|\nabla\rho|^2+\frac{\left(\sm\theta_{\frac{1}{\ve}}(\sr)+\sr\theta_{\frac{1}{\ve}}(\sm)\right)\theta_1(\sqrt{\rho\mu})}{\rho\mu}\nabla\rho\cdot\nabla\mu+\frac{1+\ve}{\mu}|\nabla\mu|^2\right)dx\nonumber\\
	&=&-\io\left(\frac{1+\ve}{\rho}|\nabla\rho|^2+\frac{S_\ve\theta_1(\sqrt{\rho\mu})}{\sqrt{\rho\mu}}\nabla\rho\cdot\nabla\mu+\frac{1+\ve}{\mu}|\nabla\mu|^2\right)dx\nonumber\\
	&=&-\io\left|\frac{\sqrt{1+\ve}}{\sr}\nabla\rho+\frac{S_\ve\theta_1(\sqrt{\rho\mu})}{2\sqrt{(1+\ve)}\sqrt{\mu}}\nabla\mu\right|^2dx\nonumber\\
	&&-\io\frac{1}{4(1+\ve)\mu}\left(4(1+\ve)^2-S_\ve^2\theta_1^2(\sqrt{\rho\mu})\right)\left|\nabla\mu\right|^2dx.
\end{eqnarray}
Subsequently,
\begin{eqnarray}
	\lefteqn{\sup_{0\leq t\leq T}\io\io(\rho\ln\rho+\mu\ln\mu)dx}\nonumber\\
	&&+\frac{1}{1+\ve}\ioT\left|2(1+\ve)\nabla\sqrt{\rho}+S_\ve\theta_1(\sqrt{\rho\mu})\nabla\sqrt{\mu}\right|^2dxdt\nonumber\\
	&&+\frac{1}{1+\ve}\ioT\left[4(1+\ve)^2-S_\ve^2\theta_1^2(\sqrt{\rho\mu})\right]\left|\nabla\sqrt{\mu}\right|^2dx\nonumber\\
	&\leq& 2\io\left(\rho(x,0)\ln\rho(x,0)+\mu(x,0)\ln\mu(x,0)\right)dx\leq c.\label{ho10}
\end{eqnarray}
Here and in what follows the letter $c$ denotes a generic positive number. The last step in \eqref{ho10} is due to \eqref{appx4} and (H2).
Similarly,
\begin{eqnarray}
	\lefteqn{\frac{1+\ve}{\rho}|\nabla\rho|^2+\frac{\left(\sm\theta_{\frac{1}{\ve}}(\sr)+\sr\theta_{\frac{1}{\ve}}(\sm)\right)\theta_1(\sqrt{\rho\mu})}{\rho\mu}\nabla\rho\cdot\nabla\mu+\frac{1+\ve}{\mu}|\nabla\mu|^2}\nonumber\\
	&=&\left|\frac{\sqrt{1+\ve}}{\sm}\nabla\mu+\frac{S_\ve\theta_1(\sqrt{\rho\mu})}{2\sqrt{(1+\ve)}\sqrt{\rho}}\nabla\rho\right|^2+\frac{1}{4(1+\ve)\rho}\left(4(1+\ve)^2-S_\ve^2\theta_1^2(\sqrt{\rho\mu})\right)\left|\nabla\rho\right|^2.
	%\frac{(1+\ve)^2-\theta^2(\sqrt{\rho\mu})}{(1+\ve)\rho}\left|\nabla\rho\right|^2.
\end{eqnarray}
Therefore,
\begin{eqnarray}
	\lefteqn{\sup_{0\leq t\leq T}\io(\rho\ln\rho+\mu\ln\mu)dx}\nonumber\\
	&&+\frac{1}{1+\ve}\ioT\left|2(1+\ve)\nabla\sqrt{\mu}+S_\ve\theta_1(\sqrt{\rho\mu})\nabla\sqrt{\rho}\right|^2dxdt\nonumber\\
	&&+\frac{1}{1+\ve}\ioT\left[4(1+\ve)^2-S_\ve^2\theta_1^2(\sqrt{\rho\mu})\right]\left|\nabla\sqrt{\rho}\right|^2dx\leq c.
	%\nonumber\\
%	&\leq& 2\io\left[\left(\sqrt{\rho_0(x)}+\ve\right)^2\ln\left(\sqrt{\rho_0(x)}+\ve\right)^2+\left(\sqrt{\mu_0(x)}+\ve\right)^2\ln\left(\sqrt{\mu_0(x)}+\ve\right)^2\right]dx.
\label{ho9}
\end{eqnarray}
Let $u,v$ be given as in \eqref{ho0}.
%Introduce
%$$u_\ve=\sqrt{\rho}+\sqrt{\mu},\ \ v_\ve=\sqrt{\rho}-\sqrt{\mu}.$$
We can derive from \eqref{eli}, \eqref{ho9}, and \eqref{ho10} that
\begin{eqnarray}
	\lefteqn{\sup_{0\leq t\leq T}\io(\rho\ln\rho+\mu\ln\mu)dx}\nonumber\\
	&&+\frac{1}{4(1+\ve)}\ioT\left|\left(2(1+\ve)+S_\ve\theta_1(\sqrt{\rho\mu})\right)\nabla u\right|^2dxdt\nonumber\\
	&&+\frac{1}{4(1+\ve)}\ioT\left|\left(2(1+\ve)-S_\ve\theta_1(\sqrt{\rho\mu})\right)\nabla v\right|^2dxdt\nonumber\\
	&&+\frac{1}{2(1+\ve)}\ioT\left[4(1+\ve)^2-S_\ve^2\theta_1^2(\sqrt{\rho\mu})\right]\left(\left|\nabla\sqrt{\rho}\right|^2+\left|\nabla\sqrt{\mu}\right|^2\right)dx\leq c.
%	&\leq&2\io\left[\left(\sqrt{\rho_0(x)}+\ve\right)^2\ln\left(\sqrt{\rho_0(x)}+\ve\right)^2+\left(\sqrt{\mu_0(x)}+\ve\right)^2\ln\left(\sqrt{\mu_0(x)}+\ve\right)^2\right]dx.
\label{pl8}
\end{eqnarray}

We are ready to design an approximate scheme to system \eqref{appx1}-\eqref{appx4}. As in \cite{CDJ,J}, we introduce the  entropy variables
\begin{equation}\label{ent5}
	\wo=\psi_\rho(\rho,\mu)=\ln\rho+1,\ \ \ \wt=\psi_\mu(\rho,\mu)=\ln\mu+1.
\end{equation}
Solve the above system for $\rho,\mu$ and then substitute them into \eqref{appx1}-\eqref{appx4} to derive
\begin{eqnarray}
	\pt e^{\wo-1}&=&\mbox{div}\left((1+\ve)e^{\wo-1}\nabla\wo+\theta_{\frac{1}{\ve}}(e^{\frac{1}{2}(\wo-1)})e^{\frac{1}{2}(\wt-1)}\theta_1\left(e^{\frac{1}{2}(\wo+\wt)-1}\right)\nabla\wt\right),\label{ent1}\\
	\pt e^{\wt-1}&=&\mbox{div}\left((1+\ve)e^{\wt-1}\nabla\wt+\theta_{\frac{1}{\ve}}(e^{\frac{1}{2}(\wt-1)})e^{\frac{1}{2}(\wo-1)}\theta_1\left(e^{\frac{1}{2}(\wo+\wt)-1}\right)\nabla\wo\right),\label{ent2}\\
	\nabla\wo\cdot\mathbf{n}&=&	\nabla\wt\cdot\mathbf{n}=0\ \ \mbox{on $\Sigma_T$},\label{ent3}\\
	\wo(x,0)&=&\ln\rho(x,0)+1,\ \ \wt(x,0)=\ln\mu(x,0)+1.\label{ent4}
\end{eqnarray}
Recall from \eqref{appx4} that
$$\rho(x,0)=\rho_\ve(x,0)\in \left[\ve,\frac{1}{\ve}+\ve\right],\ \ \mu(x,0)=\mu_\ve(x,0)\in \left[\ve,\frac{1}{\ve}+\ve\right].$$
Let $ T > 0 $ be given as in problem \eqref{eq1}-\eqref{eq4}. For each $ j \in \{ 1, 2, 3,... \} $ we divide the time interval $ \left[ 0, T \right] $ into $ j $ equal sub-intervals. Set
\begin{equation}
	\sigma = \frac{T}{j},\ \  t_k = k \sigma,\  k=0, 1,\cdots, j.
\end{equation}
Pick a positive integer $m$ so large that 
\begin{equation}\label{mdef}
\mbox{the inclusion $W^{m,2}(\Omega)\hookrightarrow W^{1,\infty}(\Omega)$ is continuous.}	
\end{equation}
According to the Sobolev embedding theorem, $ m>\frac{N}{2}+1$ is good enough (\cite{GT}, p.171). 
%(A suitable smooth approximation of the boundary $\ob$ may be needed here.)
Define a finite sequence $\{\mathbf{w}^{(k)}=(\wo^{(k)}, \wt^{(k)})\}_{k=0}^j$ as follows: Let
$$\mathbf{w}^{(0)}=\left(\ln\rho(x,0)+1,\ln\mu(x,0)+1\right).$$
Whenever $\mathbf{w}^{(k-1)}\equiv(\wo^{(k-1)}, \wt^{(k-1)}),\  k=1,\cdots, j$, is given, we let $\mathbf{w}^{(k)}=(\wo^{(k)}, \wt^{(k)})$ be the solution of the following problem: Find $\mathbf{w}^{(k)}\in \left(W^{m,2}(\Omega)\right)^2$ such that
%Let $ u_0 $ be given, satisfying (H3). It is not difficult to see from (H3) that $u_0\in L^{\infty}(\Omega)$. Thus by Proposition \ref{p21} we can recursively solve the system,=(\wo^{(0)}, \wt^{(0)})
\begin{eqnarray}
\lefteqn{	\frac{1}{\sigma}\io \left(e^{\wok-1} -e^{\wo^{(k-1)}-1}\right)\phi_1 + \left(e^{\wtk-1} -e^{\wt^{(k-1)}-1}\right)\phi_2dx }\nonumber\\
&&+\io\left((1+\ve)e^{\wok-1}\nabla\wok+\theta_{\frac{1}{\ve}}(e^{\frac{1}{2}(\wok-1)})e^{\frac{1}{2}(\wtk-1)}\theta_1\left(e^{\frac{1}{2}(\wok+\wtk)-1}\right)\nabla\wtk\right)\cdot\nabla\phi_1dx\nonumber\\
&&+\io\left((1+\ve)e^{\wtk-1}\nabla\wtk+\theta_{\frac{1}{\ve}}(e^{\frac{1}{2}(\wtk-1)})e^{\frac{1}{2}(\wok-1)}\theta_1\left(e^{\frac{1}{2}(\wok+\wtk)-1}\right)\nabla\wok\right)\cdot\nabla\phi_2dx \nonumber\\
&&+\sigma\io\left(\sum_{|\alpha|=m}D^\alpha \mathbf{w}^{(k)}\cdot D^\alpha \mathbf{\phi}+\mathbf{w}^{(k)}\cdot\mathbf{\phi}\right)dx=0\ \ \mbox{for each $\mathbf{\phi}=(\phi_1,\phi_2)\in \left(W^{m,2}(\Omega)\right)^2$.}\label{td3} 
\end{eqnarray}
Here we have adopted the notations in \cite{J}. Obviously, problem \eqref{td3} represents an implicit discretization in the time variable for system \eqref{ent1}-\eqref{ent2} with a regularizing term added.  In view of \eqref{ent5} and the proof of \eqref{pl8}, we are in a position to apply
 Lemma 5 in \cite{J}. Upon doing so, we obtain that there is a solution to \eqref{td3}.

Next, we set
$$\rho_k=e^{\wok-1}, \ \ \mu_k=e^{\wtk-1}.$$
Then define the functions $	 \tilde{\rho}_j\left( x, t \right),\bar{\rho}_j \left( x, t \right), \tilde{\mu}_j \left( x, t \right), 	\bar{\mu}_j \left( x, t \right), \bar{\mathbf{w}}^{(j)}\left( x, t \right)=(\bar{w}_1^{(j)} \left( x, t \right),\bar{w}_2^{(j)} \left( x, t \right))$ on $ \Omega_T$ as follows: For each $(x,t)\in \Omega_T$
there is $k\in\{1,\cdots,j\}$ such that $t \in (t_{k-1}, t_k] $. Subsequently, set
\begin{gather}
	\tilde{\rho}_j \left( x, t \right) = \frac{ t - t_{k-1} }{\sigma} \rho_k \left( x \right) + \left( 1 - \frac{ t - t_{k-1} }{\sigma} \right) \rho_{k-1} \left( x \right),  \label{usquig} \\
	\bar{\rho}_j \left( x, t \right) = \rho_k \left( x \right), \ \bar{w}_1^{(j)} \left( x, t \right) = \wok \left( x \right),\label{ubar} \\
	\tilde{\mu}_j \left( x, t \right) = \frac{ t - t_{k-1}}{\sigma} \mu_k \left( x \right) + \left( 1 - \frac{ t - t_{k-1}}{\sigma} \right) \mu_{k-1} \left( x \right) \label{wsquig} \\
	\bar{\mu}_j \left( x, t \right) = \mu_k \left( x \right), \ \bar{w}_2^{(j)} \left( x, t \right) = \wtk \left( x \right). \label{wbar} %\ \ \mbox{for $ x \in \Omega$, and $ .}
\end{gather}
Using these functions, we may then write our discretized problem \eqref{td3} as
\begin{eqnarray}
\lefteqn{	\partial_t \tilde{\rho}_j -\mbox{div}\left[2(1+\ve)\sqrt{\rbj}\nabla\sqrt{\rbj} +2\theta_{\frac{1}{\ve}}(\sqrt{\bar{\rho}_j})\theta_1(\sqrt{\bar{\rho}_j\bar{\mu}_j}) \nabla\sqrt{\bar{\mu}_j}\right]}\nonumber\\
&=&  \sigma \left(\sum_{|\alpha|=m}D^{2\alpha}\bar{w}_1^{(j)}+\bar{w}_1^{(j)} \right)\;\; \mbox{ on } \; \Omega_T, \label{td1} \\
\lefteqn{\partial_t \tilde{\mu}_j -\mbox{div}\left[2(1+\ve)\sqrt{\mbj}\nabla\sqrt{\bar{\mu}_j} +2\theta_{\frac{1}{\ve}}(\sqrt{\bar{\mu}_j})\theta_1(\sqrt{\bar{\rho}_j\bar{\mu}_j}) \nabla\sqrt{\bar{\rho}_j}\right]}\nonumber\\
&=&  \sigma \left(\sum_{|\alpha|=m}D^{2\alpha}\bar{w}_2^{(j)}+\bar{w}_2^{(j)} \right)	\; \mbox{ on $\oT$.}  \label{td2} 
\end{eqnarray}
(Here we have omitted the associated boundary conditions.)
We now proceed to derive the discrete analogues of our a-priori estimates. For this purpose,
set
$$S_k=\frac{\theta_{\frac{1}{\ve}}(e^{\frac{1}{2}(\wok-1)})}{e^{\frac{1}{2}(\wok-1)}}+\frac{\theta_{\frac{1}{\ve}}(e^{\frac{1}{2}(\wtk-1)})}{e^{\frac{1}{2}(\wtk-1)}}=\frac{\theta_{\frac{1}{\ve}}(\sqrt{\rk})}{\sqrt{\rk}}+\frac{\theta_{\frac{1}{\ve}}(\sqrt{\mk})}{\sqrt{\mk}}.$$
Then we calculate that
\begin{eqnarray}
\lefteqn{	\io\left((1+\ve)e^{\wok-1}\nabla\wok+\theta_{\frac{1}{\ve}}(e^{\frac{1}{2}(\wok-1)})e^{\frac{1}{2}(\wtk-1)}\theta_1\left(e^{\frac{1}{2}(\wok+\wtk)-1}\right)\nabla\wtk\right)\cdot\nabla \wok dx}\nonumber\\
	&&+\io\left((1+\ve)e^{\wtk-1}\nabla\wtk+\theta_{\frac{1}{\ve}}(e^{\frac{1}{2}(\wtk-1)})e^{\frac{1}{2}(\wok-1)}\theta_1\left(e^{\frac{1}{2}(\wok+\wtk)-1}\right)\nabla\wok\right)\cdot\nabla \wtk dx \nonumber\\
	&=&\io\left((1+\ve)e^{\wok-1}|\nabla\wok|^2+(1+\ve)e^{\wtk-1}|\nabla\wtk|^2 \right)dx\nonumber\\
	&&+\io e^{\frac{1}{2}(\wok+\wtk)-1} S_k\theta_1\left(e^{\frac{1}{2}(\wok+\wtk)-1}\right)\nabla\wtk\cdot\nabla \wok dx\nonumber\\
	&=&\frac{1}{1+\ve}\io\left|(1+\ve)e^{\frac{\wok-1}{2}}\nabla\wok+\frac{1}{2}e^{\frac{1}{2}(\wtk-1)} S_k\theta_1\left(e^{\frac{1}{2}(\wok+\wtk)-1}\right)\nabla\wtk\right|^2dx\nonumber\\
	&&+\frac{1}{4(1+\ve)}\io e^{\wtk-1}\left[4(1+\ve)^2-S_k^2\theta_1^2\left(e^{\frac{1}{2}(\wok+\wtk)-1}\right)\right]|\nabla\wtk|^2dx\nonumber\\
	&=&\frac{1}{1+\ve}\io\left|2(1+\ve)\nabla\sqrt{\rho_k}+S_k\theta_1\left(\sqrt{\rk\mk}\right)\nabla\sqrt{\mk}\right|^2dx\nonumber\\
	&&+\frac{1}{1+\ve}\io\left[4(1+\ve)^2-S_k^2\theta_1^2\left(\sqrt{\rk\mk}\right)\right]|\nabla\sqrt{\mk}|^2dx.\label{ent7}
\end{eqnarray}%In view of \eqref{ent5}, we can conclude from Lemma 5 in \cite{J} that there is a solution to the problem.
It is elementary for us to verify that
$$\left(e^{s_1-1}-e^{s_2-1}\right)s_1\geq e^{s_1-1}s_1-e^{s_2-1}s_2-\left(e^{s_1-1}-e^{s_2-1}\right).$$
Subsequently,
\begin{eqnarray}
	\lefteqn{	\frac{1}{\sigma}\io \left(e^{\wok-1} -e^{\wo^{(k-1)}-1}\right)\wok + \left(e^{\wtk-1} -e^{\wt^{(k-1)}-1}\right)\wtk dx }\nonumber\\
	&\geq&\frac{1}{\sigma}\left(\io\left(e^{\wok-1}\wok-e^{\wo^{(k-1)}-1}\wo^{(k-1)}\right)dx-\io\left(e^{\wok-1}-e^{\wo^{(k-1)}-1}\right)dx\right)\nonumber\\
	&&+\frac{1}{\sigma}\left(\io\left(e^{\wtk-1}\wtk-e^{\wt^{(k-1)}-1}\wt^{(k-1)}\right)dx-\io\left(e^{\wtk-1}-e^{\wt^{(k-1)}-1}\right)dx\right)\nonumber\\
	&=&\frac{1}{\sigma}\io\left(\rk\ln\rk-\rho_{k-1}\ln\rho_{k-1}+\mk\ln\mk-\mu_{k-1}\ln\mu_{k-1}\right)dx
\end{eqnarray}
With this and \eqref{ent7} in mind, we let $\mathbf{\phi}=\mathbf{w}^{(k)}$ in \eqref{td3} to deduce
\begin{eqnarray}
	\lefteqn{\frac{1}{\sigma}\io\left(\rk\ln\rk-\rho_{k-1}\ln\rho_{k-1}+\mk\ln\mk-\mu_{k-1}\ln\mu_{k-1}\right)dx}\nonumber\\
	&&+\frac{1}{1+\ve}\io\left|2(1+\ve)\nabla\sqrt{\rho_k}+S_k\theta_1\left(\sqrt{\rk\mk}\right)\nabla\sqrt{\mk}\right|^2dx\nonumber\\
	&&+\frac{1}{1+\ve}\io\left[4(1+\ve)^2-S_k^2\theta_1^2\left(\sqrt{\rk\mk}\right)\right]|\nabla\sqrt{\mk}|^2dx\nonumber\\
	&&+\sigma\io\left(\sum_{|\alpha|=m}|D^{\alpha}\mathbf{w}^{(k)}|^2+|\mathbf{w}^{(k)}|^2\right)dx
	\leq 0.
\end{eqnarray}
Multiply through the above inequality by $\sigma$ and sum up the resulting inequalities over $k$ to obtain
\begin{eqnarray}
\lefteqn{	\sup_{0\leq t\leq T}\io (\rbj\ln\rbj+\mbj\ln\mbj)dx}\nonumber\\
&&+\ioT\left|2(1+\ve)\nabla\sqrt{\rbj}+	\bar{S}_j\theta_1\left(\sqrt{\rbj\mbj}\right)\nabla\sqrt{\mbj}\right|^2dxdt\nonumber\\
&&+\ioT\left[4(1+\ve)^2-\bar{S}_j^2\theta_1^2\left(\sqrt{\rbj\mbj}\right)\right]|\nabla\sqrt{\mbj}|^2dxdt\nonumber\\
&&+\sigma\ioT\left(\sum_{|\alpha|=m}|D^{\alpha}\mathbf{w}^{(j)}|^2+|\mathbf{w}^{(j)}|^2\right)dxdt\nonumber\\
&\leq &c\io(\rho(x,0)\ln \rho(x,0)+\mu(x,0)\ln\mu(x,0))dx\leq c,\label{apl7}
\end{eqnarray}
where
\begin{equation}\label{sbd}
	\bar{S}_j=\frac{\theta_{\frac{1}{\ve}}(\sqrt{\rbj})}{\sqrt{\rbj}}+\frac{\theta_{\frac{1}{\ve}}(\sqrt{\mbj})}{\sqrt{\mbj}}.
\end{equation}
In fact, it is not difficult to see from \eqref{ent7} and the proof of \eqref{pl8} that
\begin{eqnarray}
	\lefteqn{\sup_{0\leq t\leq T}\io(\rbj\ln\rbj+\mbj\ln\mbj)dx}\nonumber\\
	&&+\ioT\left|\left(2(1+\ve)+\bar{S}_j\theta_1\left(\sqrt{\rbj\mbj}\right)\right)\nabla \left(\sqrt{\rbj}+\sqrt{\mbj}\right)\right|^2dxdt\nonumber\\
	&&+\ioT\left|\left(2(1+\ve)-\bar{S}_j\theta_1\left(\sqrt{\rbj\mbj}\right)\right)\nabla \left(\sqrt{\rbj}-\sqrt{\mbj}\right)\right|^2dxdt\nonumber\\
	&&+\ioT\left[4(1+\ve)^2-\bar{S}_j^2\theta^2_1\left(\sqrt{\rbj\mbj}\right)\right]\left(\left|\nabla\sqrt{\rbj}\right|^2+\left|\nabla\sqrt{\mbj}\right|^2\right)dx\nonumber\\
	&&+\sigma\ioT\left(\sum_{|\alpha|=m}|D^{\alpha}\mathbf{w}^{(j)}|^2+|\mathbf{w}^{(j)}|^2\right)dxdt\leq c.
	%	&\leq&2\io\left[\left(\sqrt{\rho_0(x)}+\ve\right)^2\ln\left(\sqrt{\rho_0(x)}+\ve\right)^2+\left(\sqrt{\mu_0(x)}+\ve\right)^2\ln\left(\sqrt{\mu_0(x)}+\ve\right)^2\right]dx.
	\label{apl8}
\end{eqnarray}
\begin{lem}\label{pre1}For each fixed $\ve$ the sequence $\{(\rbj,\mbj)\}$ is precompact in $\left(L^1(\oT)\right)^2$.
	\end{lem}
\begin{proof}We wish to apply Theorem 1 in \cite{DJ}. First,  we easily see from \eqref{apl8} that
	\begin{equation}\label{apl1}
		\ioT|\nabla\rbj|dxdt=2\ioT\sqrt{\rbj}\left|\nabla\sqrt{\rbj}\right|dxdt\leq 2\left(\ioT\rbj\right)^{\frac{1}{2}}\left(\ioT\left|\nabla\sqrt{\rbj}\right|^2dxdt\right)^{\frac{1}{2}}\leq c.
	\end{equation}
Similarly,
\begin{equation}
	\ioT|\nabla\mbj|dxdt\leq c.
\end{equation}
Combining the preceding two estimates with \eqref{mdef} and \eqref{td1} yields
\begin{eqnarray}
	\|\pt\rtj\|_{\left(W^{m,2}(\Omega)\right)^*}&\leq& c\left\|\sqrt{\rbj}\nabla\sqrt{\rbj}\right\|_{L^1(\Omega)}+c\left\|\sqrt{\rbj}\nabla\sqrt{\mbj}\right\|_{L^1(\Omega)}\nonumber\\
	&&+\sigma\left(\sum_{|\alpha|=m}\left\|D^\alpha\wob\right\|_{L^2(\Omega)}+\left\|\wob\right\|_{L^2(\Omega)}\right).
\end{eqnarray}
Consequently,
\begin{eqnarray}
\frac{1}{\sigma}	\int_{\sigma}^{T}\|\rbj(x,t)-\rbj(x,t-\sigma)\|_{\left(W^{m,2}(\Omega)\right)^*}dt&=&\frac{1}{\sigma}\sum_{k=2}^{j}	\int_{t_{k-1}}^{t_k}\|\rk(x)-\rho_{k-1}(x)\|_{\left(W^{m,2}(\Omega)\right)^*}dt\nonumber\\
&\leq &\int_{\sigma}^{T}\|\pt\rtj\|_{\left(W^{m,2}(\Omega)\right)^*}\leq c.
\end{eqnarray}
Here we have used \eqref{apl7}. Obviously, we have
$$W^{1,1}(\Omega)\hookrightarrow L^1(\Omega)\hookrightarrow \left(W^{m,2}(\Omega)\right)^*$$
with the first inclusion being compact. We are ready to conclude from Theorem 1 in \cite{DJ} that $\{\rbj\}$ is precompact in $L^1(\oT)$. The same argument can be applied to  $\{\mbj\}$. The proof is complete. \end{proof}
\begin{lem}\label{pre2}The sequence $\{\rtj,\mtj\}$ is precompact in $\left(L^1(\oT)\right)^2$.
	\end{lem}
\begin{proof}
	First, we observe from \eqref{usquig} that for $t\in (t_{k-1}, t_k]$ we have
\begin{eqnarray}
	\left|\nabla\sqrt{\rtj}\right|&=&\left|\frac{\frac{t-t_{k-1}}{\sigma}\sqrt{\rk(x)}\nabla\sqrt{\rk(x)}+\frac{t_k-t}{\sigma}\sqrt{\rho_{k-1}(x)}\nabla\sqrt{\rho_{k-1}(x)}}{\sqrt{\frac{t-t_{k-1}}{\sigma}\rk(x)+\frac{t_k-t}{\sigma}\rho_{k-1}(x)}}\right|\nonumber\\
	&\leq&\sqrt{\frac{t-t_{k-1}}{\sigma}}\left|\nabla\sqrt{\rk(x)}\right|+\sqrt{\frac{t_k-t}{\sigma}}\left|\nabla\sqrt{\rho_{k-1}(x)}\right|.
\end{eqnarray}
As a result, we have
%we calculate from \eqref{usquig} and \eqref{apl1} that\left|\frac{t-t_{k-1}}{\sigma}\nabla\rk(x)+\frac{t_k-t}{\sigma}\nabla\rho_{k-1}(x)\right|
\begin{eqnarray}
	\ioT|\nabla\sqrt{\rtj}|^2dxdt&=&\sum_{k=1}^{j}\int_{t_{k-1}}^{t_k}\left(\sqrt{\frac{t-t_{k-1}}{\sigma}}\left|\nabla\sqrt{\rk(x)}\right|+\sqrt{\frac{t_k-t}{\sigma}}\left|\nabla\sqrt{\rho_{k-1}(x)}\right|\right)^2dxdt\nonumber\\
	&\leq&\sigma\sum_{k=1}^{j}\io|\nabla\sqrt{\rk(x)}|^2dx+\sigma\sum_{k=1}^{j}\io|\nabla\sqrt{\rho_{k-1}(x)}|^2dx\nonumber\\
	&=&\ioT\left|\nabla\sqrt{\rbj}\right|^2dxdt+\int_{0}^{T-\sigma}\io\left|\nabla\sqrt{\rbj}\right|^2dxdt+\sigma\io\left|\nabla\sqrt{\rho_0(x)}\right|^2dx\leq c.
\end{eqnarray}
Similarly,
\begin{equation}
	\ioT\rtj dxdt\leq c.
\end{equation}
It is easy to see from \eqref{td1} that $\{\pt\rtj\}$ is bounded in $L^1(0,T; \left(W^{m,2}(\Omega)\right)^*)$. We are in a position to use Lemma \ref{la}. Upon doing so, we obtain
\begin{equation}
	\mbox{$\{\rtj\}$ is precompact in $L^1(\oT)$.}
\end{equation}
By the same token, we can show 
\begin{equation}
	\mbox{$\{\mtj\}$ is precompact in $L^1(\oT)$.}
\end{equation}
The proof is complete.
\end{proof}
It is easy to check 
\begin{equation}
	\rtj-	\rbj=\frac{t-t_k}{\sigma}(\rk(x)-\rho_{k-1}(x))=(t-t_k)\pt\rtj\ \ \mbox{for $t\in (t_{k-1}, t_k]$}.
\end{equation}
%For each $\vp\in L^\infty(\Omega)$ we calculate from \eqref{usquig} and \eqref{ubar} that
Therefore,
\begin{eqnarray}
	\ioT\left\|\rbj-\rtj\right\|_{\left(W^{m,2}(\Omega)\right)^*} dt	&\leq&\sigma\ioT\left\|\pt\rtj\right\|_{\left(W^{m,2}(\Omega)\right)^*} dt\leq c\sigma.\label{apl2}
\end{eqnarray}
This together with Lemmas \ref{pre1} and \ref{pre2} enables us to assume
\begin{eqnarray*}
	(\rbj,\mbj)&\ra& (\rho,\mu)\ \  \mbox{strongly in $L^1(\oT)$ and a .e. on $\oT$ as $j\ra \infty$,}\label{apl3}\\
	(\rtj,\mtj)&\ra& (\rho,\mu)\ \  \mbox{strongly in $L^1(\oT)$ and a .e. on $\oT$ as $j\ra \infty$.}\label{apl4}
\end{eqnarray*}
We can let $\sigma\ra 0$	in system \eqref{td1}-\eqref{td2} to get \eqref{appx1} and \eqref{appx2}. We shall omit the details.

\section{Take $\ve\ra 0$}\label{sec3}
In this section, we offer the proof of our main theorem.

Take
$$\ve=\frac{1}{n},\ \  n=1,2, \cdots.$$
Let $(\re, \me)$ be the solution to \eqref{appx1}-\eqref{appx4} constructed in Section \ref{sec2}. Obviously, \eqref{pl8} still remains valid. That is, we have
\begin{eqnarray}
	\lefteqn{\sup_{0\leq t\leq T}\io(\re\ln\re+\me\ln\me)dx+\ioT\left|\nabla u_\ve\right|^2dxdt}\nonumber\\
	&&+\ioT\left(1+\ve-\theta_1(\sqrt{\re\me})\right)^2\left|\nabla v_\ve\right|^2dxdt\nonumber\\
	&&+\ioT\left[4(1+\ve)^2-S_\ve^2\theta^2(\sqrt{\re\me})\right]\left(\left|\nabla\sqrt{\re}\right|^2+\left|\nabla\sqrt{\me}\right|^2\right)dx\leq c,
	%	&\leq&2\io\left[\left(\sqrt{\rho_0(x)}+\ve\right)^2\ln\left(\sqrt{\rho_0(x)}+\ve\right)^2+\left(\sqrt{\me_0(x)}+\ve\right)^2\ln\left(\sqrt{\mu_0(x)}+\ve\right)^2\right]dx.
	\label{pl88}
\end{eqnarray}
where as in \eqref{sdef} we have
\begin{equation}\label{sdef2}
	S_\ve=\frac{\theta_{\frac{1}{\ve}}(\sre)}{\sre}+\frac{\theta_{\frac{1}{\ve}}(\sme)}{\sme}\leq 2.
\end{equation}
We substitute $u_\ve=\sqrt{\re}+\sqrt{\me}$ into \eqref{appx1} and \eqref{appx2} to derive
\begin{eqnarray}
	\pt\re&=&2\mbox{div}\left[\left((1+\ve)\sre-\theta_{\frac{1}{\ve}}(\sqrt{\re})\theta_1(\sqrt{\re\me})\right)\nabla\sre+\theta_{\frac{1}{\ve}}(\sqrt{\re})\theta_1(\sqrt{\re\me})\nabla u_\ve\right]\nonumber\\
	&& \ \mbox{in $\oT$,}\label{pl55}\\
	\pt\me&=&2\mbox{div}\left[((1+\ve)\sme-\theta_{\frac{1}{\ve}}(\sme)\theta_1(\sqrt{\re\me}))\nabla\sme+\theta_{\frac{1}{\ve}}(\sqrt{\me})\theta_1(\sqrt{\re\me})\nabla u_\ve\right]\nonumber\\
	&&\ \mbox{in $\oT$.}\label{pl66}
\end{eqnarray}
Our objective is to show that we can pass to the limit in the preceding system. To this end, we pick
\begin{eqnarray}
	a\in W^{3,\infty}_0(-\infty,1).\label{mt5}
\end{eqnarray}
%Then we have
Then we have
\begin{equation}\label{apl6}
	|a(s)|+|a^\prime(s)|+|a^{\prime\prime}(s)|\leq 2\|a\|_{ W^{3,\infty}_0(-\infty,1)}(1-\theta_1(\sqrt{s}))\ \ \mbox{on $[0,\infty)$}.
\end{equation}
%According to  the mean value theorem, there is a number $s_1$ lying between $s$ and $1$ such that
%$$a^\prime(s)=a^\prime(s)-a^\prime(1)=a^{\prime\prime}(s_1)(s-1).$$
%Subsequently, 
Indeed, for $s\in [0,1]$ we estimate that
\begin{eqnarray}
|a^{\prime\prime}(s)|&=&|a^{\prime\prime}(s)-a^{\prime\prime}(1)|\nonumber\\
&\leq&\|a^{\prime\prime\prime}\|_{\infty, (-\infty,1)}(1-s)\nonumber\\
&=&\|a^{\prime\prime\prime}\|_{\infty, (-\infty,1)}(1-\sqrt{s}) (\sqrt{s}+1)\nonumber\\
&\leq& 2\|a^{\prime\prime\prime}\|_{\infty, (-\infty,1)}(1-\theta_1(\sqrt{s})).
\end{eqnarray}
The above inequality is trivially true for $s>1$. Apply the same argument to $a$ and $ a^\prime$, respectively, to get \eqref{apl6}.

Take $b,\ d\in C^2_{\textup{c}}(\mathbb{R})$. That is, $b$ and $ d$ are twice continuously differentiable with compact supports. Consequently, there is a positive number $c=c(b,d)$ such that
\begin{eqnarray}
	\sqrt{s}(1+s+s^2)(|b(s)|+|b^\prime(s)|+|b^{\prime\prime}(s)|)&\leq &c\ \ \mbox{and}\nonumber\\ \sqrt{s}(1+s+s^2)(|d(s)|+|d^\prime(s)|+|d^{\prime\prime}(s)|)&\leq& c\ \ \mbox{for all $s\in [0,\infty)$.}\label{apl10}
\end{eqnarray}
Set
$$\zeta(s,\tau)=a(s\tau)b(s)d(\tau).$$
We can easily derive from \eqref{apl6} and \eqref{apl10} that
\begin{eqnarray}
\lefteqn{(\sqrt{s\tau}+s+\tau)\left(|\zeta(s,\tau)|+|\nabla\zeta(s,\tau)|+|\nabla^2\zeta(s,\tau)|\right)}	\nonumber\\
&\leq&c\left(|a(s\tau)|+|a^\prime(s\tau)|+|a^{\prime\prime}(s\tau)|\right)\leq c(1-\theta_1(\sqrt{s\tau}))\ \ \mbox{for $(s,\tau)\in [0,\infty)\times[0,\infty)$.}\label{apl9}
\end{eqnarray}

The key to our development is the following lemma.
\begin{lem}The sequence $\{\zeta(\re,\me)\}$ is precompact in $L^p(\oT)$ for each $p\geq 1$.
\end{lem}
\begin{proof}
	We calculate
	\begin{eqnarray}
		\nabla\zeta(\re,\me)&=&\zeta_{s}(\re,\me)\nabla\re+	\zeta_{\tau}(\re,\me)\nabla\me\nonumber\\
		&=&2\sqrt{\re}\zeta_{s}(\re,\me)\nabla\sqrt{\re}+2\sqrt{\me}\zeta_{\tau}(\re,\me)\nabla\sqrt{\me}.\label{mt2}
	\end{eqnarray}
We can easily verify that
$$2\left[(1+\ve)^2-\theta^2_1(\sqrt{\re\me})\right]\geq \left(1+\ve-\theta_1(\sqrt{\re\me})\right)^2+(1+\ve)\left(1+\ve-\theta_1(\sqrt{\re\me})\right).$$
%$$|a^\prime(s)|=|a^{\prime\prime}(s_1)(\sqrt{s}-1) (\sqrt{s}+1)|\leq 2\|a^{\prime\prime}\|_\infty(1-\sqrt{s})=2\|a^{\prime\prime}\|_\infty(1-\theta_1(\sqrt{s})).$$+(1+\ve)\left(1+\ve-\theta_1(\sqrt{\re\me})\right)
With these in mind, we compute from \eqref{apl9} and \eqref{pl88} that
\begin{eqnarray}
\ioT|\nabla\zeta(\re,\me)|^2dxdt&\leq&	c\ioT|\sqrt{\re}\zeta_{s}(\re,\me)\nabla\sqrt{\re}|^2dxdt+	c\ioT|\sqrt{\me}\zeta_{\tau}(\re,\me)\nabla\sqrt{\me}|^2dxdt\nonumber\\
&\leq& c\ioT\left|(1-\theta_1\sqrt{\re\me})\nabla\sqrt{\re}\right|^2dxdt+ c\ioT\left|(1-\theta_1\sqrt{\re\me})\nabla\sqrt{\me}\right|^2dxdt\nonumber\\
	&\leq &c\ioT\left[(1+\ve)^2-\theta^2_1(\sqrt{\re\me})\right]\left|\nabla\sqrt{\re}\right|^2dxdt\nonumber\\
	&&+c\ioT\left[(1+\ve)^2-\theta^2_1(\sqrt{\re\me})\right]\left|\nabla\sqrt{\me}\right|^2dxdt\leq c.\label{mt20}
\end{eqnarray}

Next we show that $\{\pt\zeta(\re,\me)\}$ is bounded in $L^1\left(0,T; \left(W^{1,\infty}_0(\Omega)\right)^*\right)$.
%, where $m$ is given as in Section \ref{sec2}.
Formally, we compute from \eqref{appx1} and \eqref{appx2} that
\begin{eqnarray}
\lefteqn{	\pt\zeta(\re,\me)}\nonumber\\
&=&\zeta_{s}(\re,\me)\pt\re+\zeta_{\tau}(\re,\me)\pt\me\nonumber\\
	&=&\mbox{div}\left[2\zeta_{s}(\re,\me)\sre(1+\ve)\nabla\sre+2\zeta_{s}(\re,\me)\theta_{\frac{1}{\ve}}(\sre)\theta_1(\sqrt{\re\me})\nabla\sme\right]\nonumber\\
	&&-\nabla\zeta_{s}(\re,\me)\cdot\left[2\sre(1+\ve)\nabla\sre+2\theta_{\frac{1}{\ve}}(\sre)\theta_1(\sqrt{\re\me})\nabla\sme\right]\nonumber\\
		&&+\mbox{div}\left[2\zeta_{\tau}(\re,\me)\sme\left((1+\ve)\nabla\sme+\theta_1(\sqrt{\re\me})\nabla\sre\right)\right].\nonumber\\
		&&-\nabla\zeta_{\tau}(\re,\me)\cdot\left[2\sme(1+\ve)\nabla\sme+2\theta_{\frac{1}{\ve}}(\sme)\theta_1(\sqrt{\re\me})\nabla\sre\right]\nonumber\\
		&\equiv&\mbox{div}\mathbf{F}_1-f_2+\mbox{div}\mathbf{F}_3-f_4.\label{mt19}%\nonumber\\
	%	&=&d(\me)(b(\re)a^\prime(\re\me)\me+a(\re\me) b^\prime(\re))\pt\re\nonumber\\
%	&&+b(\re)(d(\me)a^\prime(\re\me)\re+a(\re\me)d^\prime(\me))\pt\me.2\left(\zeta_{\tau s}(\re,\me)\sre\nabla\sre+\zeta_{\tau\tau}\sme\nabla\sme\right)
\end{eqnarray}
We easily see from \eqref{apl9} that
$$|\mathbf{F}_1|\leq c(1-\theta_1(\sqrt{\re\me})|\nabla\sre|+c(1-\theta_1(\sqrt{\re\me})|\nabla\sme|.$$
This together with the proof of \eqref{mt20} implies
\begin{equation}
	\ioT|\mathbf{F}_1|^2dxdt\leq c.
\end{equation}
By the same token,
\begin{equation}
	\ioT|\mathbf{F}_3|^2dxdt\leq c.
\end{equation}
As for $f_2$, we have
\begin{eqnarray}
	f_2&=&4\left(\zeta_{ss}(\re,\me)\sre\nabla\sre+\zeta_{s\tau}\sme\nabla\sme\right)\cdot\left[\sre(1+\ve)\nabla\sre+\theta_{\frac{1}{\ve}}(\sre)\theta_1(\sqrt{\re\me})\nabla\sme\right]\nonumber\\
	&=&4(1+\ve)\zeta_{ss}(\re,\me)\re|\nabla\sre|^2+4\zeta_{s\tau}\sme\theta_{\frac{1}{\ve}}(\sre)\theta_1(\sqrt{\re\me})|\nabla\sme|^2\nonumber\\
	&&+4\left(\zeta_{ss}(\re,\me)\sre\theta_{\frac{1}{\ve}}(\sre)\theta_1(\sqrt{\re\me})+\sqrt{\re\me}(1+\ve)\zeta_{s\tau}\right)\nabla\sre\cdot\nabla\sme.
\end{eqnarray}
Consequently, by \eqref{apl9},
\begin{eqnarray}
	\ioT|f_2|dxdt\leq c\ioT(1-\theta_1(\sqrt{\re\me}))\left(|\nabla\sre|^2+|\nabla\sme|^2\right)dxdt\leq c.
\end{eqnarray}
Similarly, 
\begin{equation}
	\ioT|f_4|dxdt\leq c.
\end{equation}
In summary, we have shown that the right-hand side term in \eqref{mt19} is bounded in $L^2\left(0,T; \left(W^{1,2}(\Omega)\right)^*\right)+L^1(\oT)\subset L^1\left(0,T; \left(W^{1,\infty}(\Omega)\right)^*\right)$.
Unfortunately, the first step in \eqref{mt19} is not vigorous. To be precise, the terms $\zeta_{s}(\re,\me)\pt\re$ and $\zeta_{\tau}(\re,\me)\pt\me$ do not make sense. This is due to the fact that the two terms cannot be viewed as duality parings. For example, we see from \eqref{pl55} that $\pt\re\in L^1(0,T; \left(W^{1,\infty}(\Omega)\right)^*)$, while we only have that $\zeta_{\tau}(\re,\me)\in L^2(0,T; W^{1,2}(\Omega))$. We address this issue by employing mollification (\cite{EG}, p.122). To do this, pick a function $\vp\in C^\infty_0(B_1(\mathbf{0}))$ with 
$$\vp\geq 0\ \ \mbox{and $\int_{B_1(\mathbf{0})}\vp(x,t)dxdt=1$, }$$
%replace $(\re,\me)$ with $ (\rtj,\mtj)$, which are defined in \eqref{usquig} and \eqref{wsquig}, and then show that the resulting functions are bounded in $L^1(0,T; \left(W^{m,2}(\Omega)\right)^*)$ uniformly in $j, \ve$. and $z=(x,t)$ denotes a point in  $\mathbb{R}^{N+1}$
where $B_1(\mathbf{0})$ is the unit ball in $\mathbb{R}^{N+1}$. Next let
$$\vp_\sigma(x,t)=\frac{1}{\sigma^{N+1}}\vp\left(\frac{x}{\sigma},\frac{t}{\sigma}\right), \ \ \sigma>0.$$ Define
$$\ret(x,t)=\ioT\vp_\sigma(x-y,t-s)\re(y,s)dyds,\ \ \met(x,t)=\ioT\vp_\sigma(x-y,t-s)\me(y,s)dyds.$$
%\met=\vp_\sigma*\me
It immediately follows that
\begin{equation}
	(\ret, \met)\ra(\re,\me) \ \ \mbox{strongly in $L^1(\oT)$ and a.e. on $\oT$ as $\sigma\ra 0$.}
\end{equation}
Let $Q_0$ be a subdomain of $\oT$ with $\overline{Q_0}\subset\oT$. For $\sigma<\mbox{dist}(Q_0, \partial\oT)$ and $(x,t)\in Q_0$ we calculate from \eqref{apl1} that
\begin{eqnarray}
	\pt\ret(x,t)&=&\ioT\pt\vp_\sigma(x-y,t-s)\re(y,s)ddyds\nonumber\\
	&=&-\ioT\partial_s\vp_\sigma(x-y,t-s)\re(y,s)ddyds\nonumber\\
%	&=&-\ioT\vp_\sigma(x-y,t-s)\partial_s\re(y,s)ddyds\nonumber\\
	&=&-\ioT\nabla_y\vp_\sigma(x-y,t-s)\cdot\left[(1+\ve)\nabla\re+2\theta_{\frac{1}{\ve}}(\sqrt{\re})\theta_1(\sqrt{\re\me})\nabla\sme\right]dyds\nonumber\\
	&=&\mbox{div}\left[(1+\ve)\nabla\ret(x,t)+2\left(\theta_{\frac{1}{\ve}}(\sqrt{\re})\theta_1(\sqrt{\re\me})\nabla\sme\right)^{(\sigma)}\right]\ \ \mbox{in $Q_0$}.
\end{eqnarray}
In the same way, we obtain
\begin{equation*}
	\pt\met(x,t)=\mbox{div}\left[(1+\ve)\nabla\met(x,t)+2\left(\theta_{\frac{1}{\ve}}(\sqrt{\me})\theta_1(\sqrt{\re\me})\nabla\sre\right)^{(\sigma)}\right]\ \ \mbox{in $Q_0$}.
\end{equation*}
Now we are justified to calculate
\begin{eqnarray}
	\lefteqn{	\pt\zeta(\ret,\met)}\nonumber\\
	&=&\zeta_{s}(\ret,\met)\pt\ret+\zeta_{\tau}(\ret,\met)\pt\met\nonumber\\
	&=&\mbox{div}\left[(1+\ve)\zeta_{s}(\ret,\met)\nabla\ret+2\zeta_{s}(\ret,\met)\left(\theta_{\frac{1}{\ve}}(\sqrt{\re})\theta_1(\sqrt{\re\me})\nabla\sme\right)^{(\sigma)}\right]\nonumber\\
	&&-\nabla \zeta_{s}(\ret,\met)\cdot\left[(1+\ve)\nabla\ret+2\left(\theta_{\frac{1}{\ve}}(\sqrt{\re})\theta_1(\sqrt{\re\me})\nabla\sme\right)^{(\sigma)}\right]\nonumber\\
	&&+\mbox{div}\left[(1+\ve)\zeta_{\tau}(\ret,\met)\nabla\met+2\zeta_{\tau}(\ret,\met)\left(\theta_{\frac{1}{\ve}}(\sqrt{\me})\theta_1(\sqrt{\re\me})\nabla\sre\right)^{(\sigma)}\right]\nonumber\\
	&&-\nabla \zeta_{\tau}(\ret,\met)\cdot\left[(1+\ve)\nabla\met+2\left(\theta_{\frac{1}{\ve}}(\sqrt{\me})\theta_1(\sqrt{\re\me})\nabla\sre\right)^{(\sigma)}\right]\nonumber\\
&&\ \ \ \mbox{in $Q_0$}.\label{mt17}
	%\nonumber\\
%	&\equiv&\mbox{div}\mathbf{F}_1-f_2+\mbox{div}\mathbf{F}_3-f_4.\label{mt19}%\nonumber\\
	%	&=&d(\me)(b(\re)a^\prime(\re\me)\me+a(\re\me) b^\prime(\re))\pt\re\nonumber\\
	%	&&+b(\re)(d(\me)a^\prime(\re\me)\re+a(\re\me)d^\prime(\me))\pt\me.2\left(\zeta_{\tau s}(\re,\me)\sre\nabla\sre+\zeta_{\tau\tau}\sme\nabla\sme\right)
\end{eqnarray}
We can easily take $\sigma\ra 0$ on the left-hand side of the above equation. We can do the same in the last four terms.
 %we must prove
%\begin{equation}
%	\int_{Q_0}\left|\nabla\sqrt{\ret}\right|^2dxdt+	\int_{Q_0}\left|\nabla\sqrt{\met}\right|^2dxdt\leq c(\ve).
%\end{equation}
To see this, we calculate
\begin{eqnarray}
\lefteqn{\left|	\nabla\sqrt{\ret(x,t)}\right|}\nonumber\\
&=&\left|\frac{1}{2\sqrt{\ret}}\nabla\ioT\vp_\sigma(x-y, t-s)\re(y,s)dyds\right|\nonumber\\
	&=&\left|\frac{1}{2\sqrt{\ret}}\ioT\vp_\sigma(x-y, t-s)\nabla\re(y,s)dyds
\right|\nonumber\\
&=&\left|\frac{1}{\sqrt{\ret}}\ioT\vp_\sigma(x-y, t-s)\sqrt{\re(y,s)}\nabla\sqrt{\re(y,s)}dyds
\right|\nonumber\\
&\leq&\frac{1}{\sqrt{\ret}}\left(\ioT\vp_\sigma(x-y, t-s)\re(y,s)dyds\right)^{\frac{1}{2}}\nonumber\\
&&\cdot\left(\ioT\vp_\sigma(x-y, t-s)|\nabla\sqrt{\re(y,s)}|^2dyds\right)^{\frac{1}{2}}
\nonumber\\
&=&\sqrt{\left(|\nabla\sqrt{\re}|^2\right)^{(\sigma)}(x,t)}.\label{mt18}
\end{eqnarray}
(Here we have assumed that $\re$ is bounded away from 0 below. If not, replace $\re$ by $\re+\delta$ with $\delta>0$ in the above calculations and then let $\delta\ra 0$.) It follows from \eqref{mt18} that
$$\int_{Q_0}\left|	\nabla\sqrt{\ret(x,t)}\right|^2dxdt\leq\int_{Q_0}\left(|\nabla\sqrt{\re}|^2\right)^{(\sigma)}(x,t)dxdt\leq\ioT|\nabla\sqrt{\re}|^2dyds.$$
The last step is due an inequality in (\cite{EG}, p.124). Hence, we may assume
%We easily see from \eqref{mt18} that
$$\nabla\sqrt{\ret}\ra\nabla\sqrt{\re}\ \ \mbox{weakly in $\left(L^2(Q_0)\right)^N$ as $\sigma\ra 0$.}$$
Recall that the norm in a Banach space is always weakly lower semi-continuous. Subsequently,
\begin{eqnarray}
	\int_{Q_0}|\nabla\sre|^2dxdt&\leq&\liminf_{\sigma\ra 0}\int_{Q_0}\left|\nabla\sqrt{\ret}\right|^2dxdt\nonumber\\
	&\leq &\liminf_{\sigma\ra 0}\int_{Q_0}\left(|\nabla\sqrt{\re}|^2\right)^{(\sigma)}dxdt=\int_{Q_0}|\nabla\sre|^2dxdt,
\end{eqnarray}
from whence follows
%We can infer from \eqref{mt18} that
$$\nabla\sqrt{\ret}\ra\nabla\sqrt{\re}\ \ \mbox{strongly in $\left(L^2(Q_0)\right)^N$ as $\sigma\ra 0$.}$$
The very same argument yields
 $$\nabla\sqrt{\met}\ra\nabla\sqrt{\me}\ \ \mbox{strongly in $\left(L^2(Q_0)\right)^N$ as $\sigma\ra 0$.}$$
This is enough to justify letting $\sigma\ra 0$ in each of the last four terms in \eqref{mt17}. Indeed, it is fairly straightforward to take the limit in the two divergence terms. As for the second term on the right-hand side of \eqref{mt17}, we have
\begin{eqnarray}
	\lefteqn{\nabla \zeta_{s}(\ret,\met)\cdot\left[(1+\ve)\nabla\ret+2\left(\theta_{\frac{1}{\ve}}(\sqrt{\re})\theta_1(\sqrt{\re\me})\nabla\sme\right)^{(\sigma)}\right]}\nonumber\\
	&=&4(1+\ve)\ret\zeta_{ss}(\ret,\met)\left|\nabla\sqrt{\ret}\right|^2\nonumber\\
	&&+4(1+\ve)\sqrt{\ret\met}\zeta_{s\tau}(\ret,\met)\nabla\sqrt{\ret}\cdot\nabla\sqrt{\met}\nonumber\\
	&&+4\sqrt{\ret}\zeta_{ss}(\ret,\met)\nabla\sqrt{\ret}\cdot\left(\theta_{\frac{1}{\ve}}(\sqrt{\re})\theta_1(\sqrt{\re\me})\nabla\sme\right)^{(\sigma)}\nonumber\\
&&	+4\sqrt{\met}\zeta_{s\tau}(\ret,\met)\nabla\sqrt{\met}\cdot\left(\theta_{\frac{1}{\ve}}(\sqrt{\re})\theta_1(\sqrt{\re\me})\nabla\sme\right)^{(\sigma)}.\label{mt21}
\end{eqnarray}
It is very important to note that
$$\left(\theta_{\frac{1}{\ve}}(\sqrt{\re})\theta_1(\sqrt{\re\me})\nabla\sme\right)^{(\sigma)}\ra\theta_{\frac{1}{\ve}}(\sqrt{\re})\theta_1(\sqrt{\re\me})\nabla\sme\ \ \mbox{strongly in $\left(L^2(Q_0)\right)^N$ as $\sigma\ra 0$}$$
because the other two vectors in the last two dot products in \eqref{mt21}
% $\left\{\nabla\sqrt{\ret}\right\}$ and $\left\{\nabla\sqrt{\met}\right\}$ 
are only bounded in $\left(L^2(Q_0)\right)^N$. In fact, this is why we must approximate $\sre$
by $\theta_{\frac{1}{\ve}}(\sqrt{\re})$ in the first place. The last term in \eqref{mt17} is similar to the second term. Thus, we can complete taking $\sigma\ra 0$ in the right hand side of \eqref{mt17}.

Since $Q_0$ is an arbitrary subdomain of $\oT$, we have
\begin{eqnarray}
\lefteqn{	\pt\zeta(\re,\me)}\nonumber\\
	&=&\mbox{div}\left[2\zeta_{s}(\re,\me)\sre(1+\ve)\nabla\sre+2\zeta_{s}(\re,\me)\theta_{\frac{1}{\ve}}(\sre)\theta_1(\sqrt{\re\me})\nabla\sme\right]\nonumber\\
&&-\nabla\zeta_{s}(\re,\me)\cdot\left[2\sre(1+\ve)\nabla\sre+2\theta_{\frac{1}{\ve}}(\sre)\theta_1(\sqrt{\re\me})\nabla\sme\right]\nonumber\\
&&+\mbox{div}\left[2\zeta_{\tau}(\re,\me)\sme\left((1+\ve)\nabla\sme+\theta_1(\sqrt{\re\me})\nabla\sre\right)\right].\nonumber\\
&&-\nabla\zeta_{\tau}(\re,\me)\cdot\left[2\sme(1+\ve)\nabla\sme+2\theta_{\frac{1}{\ve}}(\sme)\theta_1(\sqrt{\re\me})\nabla\sre\right]
\end{eqnarray}
in the sense of distributions in $\oT$.	Taking
	$$X_0=W^{1,2}(\Omega),\ \ X=L^2(\Omega), \ \ X_1=\left(W^{1,\infty}_0(\Omega)\right)^*,$$
	we can invoke Lemma \ref{la} to obtain 
	\begin{equation}\label{mt4}
		\mbox{$\{	b(\re)d(\me) a(\re\me) \}$ is precompact in $L^2(\oT)$.}
	\end{equation}
	This completes the proof.
\end{proof}
\begin{lem}There is a subsequence of $\{\sqrt{\re\me}\}$, still denoted by $\{\sqrt{\re\me}\}$, such that
	\begin{equation}\label{mt7}
		\mbox{$\theta_1(\sqrt{\re\me})$ converge a.e. on $\oT$.}
	\end{equation}
\end{lem}
\begin{proof} 
	We first show that
	\begin{equation}\label{acom}
		\mbox{$\{a(\re\me)\}$ is precompact in $L^1(\oT)$.}
	\end{equation}
	For each $\ell>0$ take $b$ so that
	$$b=1\ \ \mbox{on $[0,\ell]$.}$$
	With this in mind, we have
	\begin{eqnarray}
		\lefteqn{	\ioT|a(\rho_{\ve_1}\mu_{\ve_1})-a(\rho_{\ve_2}\mu_{\ve_2})|dxdt}\nonumber\\
		&\leq&\int_{\{\rho_{\ve_1}\leq\ell\}\cap\{\mu_{\ve_1}\leq\ell\}\cap\{\rho_{\ve_2}\leq\ell\}\cap\{\mu_{\ve_2}\leq\ell\}}|a(\rho_{\ve_1}\mu_{\ve_1})b(\rho_{\ve_1})b(\mu_{\ve_1})-a(\rho_{\ve_2}\mu_{\ve_2})b(\rho_{\ve_2})b(\mu_{\ve_2})|dxdt\nonumber\\
		&&+c\left|\oT\setminus \left(\{\rho_{\ve_1}\leq\ell\}\cap\{\mu_{\ve_1}\leq\ell\}\cap\{\rho_{\ve_2}\leq\ell\}\cap\{\mu_{\ve_2}\leq\ell\}\right)\right|\nonumber\\
		&\leq&\ioT|a(\rho_{\ve_1}\mu_{\ve_1})b(\rho_{\ve_1})b(\mu_{\ve_1})-a(\rho_{\ve_2}\mu_{\ve_2})b(\rho_{\ve_2})b(\mu_{\ve_2})|dxdt+\frac{c}{\ell}.
	\end{eqnarray}
The last step is due to the first term in \eqref{pl88}.
	Combining this with \eqref{mt4}  yields
	\begin{equation}\label{mt6}
		\mbox{$\{ a(\re\me) \}$ is precompact in $L^1(\oT)$.}
	\end{equation}
	Obviously, we can pick a sequence $\{a_j\}\subset W_0^{3,\infty}(-\infty,1)$ such that
	\begin{equation}\label{mtt9}
		a_j(s)\ra 1-\theta_1(s)\ \ \mbox{uniformly on $[0,\infty)$.}
	\end{equation}
	%$$$$
	(Note that both $a_j(s)$ and $1-\theta_1(s)$ are $0$ on $[1,\infty)$). By \eqref{mt6}, for each $j$ there is a subsequence of $\{a_j(\re\me)\}$ which converge a.e. on $\oT$.
	%$$a_j(\re\me) \ \ \mbox{}$$
	According to the classical diagonal argument, we can select a  subsequence of $\{ \re\me \}$, still denoted by $\{ \re\me \}$, such that for each $j$ there holds
	\begin{equation}
		a_j(\re(z)\me(z))\ \ \mbox{converges for a.e. $z\in\oT$.}\nonumber
	\end{equation}
	Combing this with \eqref{mtt9} yields
	\begin{equation}\label{mtt1}
		\theta_1(\re(z)\me(z))\ \ \mbox{converges for a.e. $z\in\oT$.}
	\end{equation}
The lemma follows from the fact that
\begin{equation}
		\theta_1(\sqrt{\re(z)\me(z)})=\sqrt{	\theta_1(\re(z)\me(z))}.
\end{equation}
 	The proof is complete.
\end{proof}
\begin{lem}\label{l33}Both $\{\sqrt{\re}(1-\theta_1(\sqrt{\re\me}))\}$ and $\{\sqrt{\me}(1-\theta_1(\sqrt{\re\me}))\}$ are precompact in $L^2(\oT)$.
\end{lem}
\begin{proof}
	Fix $\ell>0$ as before. 
	This time we pick $b, d$ so that
	$$b(s)=s,\ \ d(s)=1\ \ \mbox{on $(0,\ell)$.}$$
	For each $a\in W^{3,\infty}_0(-\infty,1)$, we obviously have $a^2\in W^{3,\infty}_0(-\infty,1)$.
	With these in mind, we calculate from \eqref{pl8} that
	\begin{eqnarray}
		\lefteqn{	\ioT|a(\rho_{\ve_1}\mu_{\ve_1})\sqrt{\rho_{\ve_1}}-a(\rho_{\ve_2}\mu_{\ve_2})\sqrt{\rho_{\ve_2}}|^2dxdt}\nonumber\\
		&\leq&	\ioT|a^2(\rho_{\ve_1}\mu_{\ve_1})\rho_{\ve_1}-a^2(\rho_{\ve_2}\mu_{\ve_2})\rho_{\ve_2}|dxdt\nonumber\\
		&\leq&\int_{\{\rho_{\ve_1}\leq\ell\}\cap\{\mu_{\ve_1}\leq\ell\}\cap\{\rho_{\ve_2}\leq\ell\}\cap\{\mu_{\ve_2}\leq\ell\}}|a^2(\rho_{\ve_1}\mu_{\ve_1})b(\rho_{\ve_1})d(\mu_{\ve_1})-a^2(\rho_{\ve_2}\mu_{\ve_2})b(\rho_{\ve_2})d(\mu_{\ve_2})|dxdt\nonumber\\
		&&+c\int_{\{\rho_{\ve_1}>\ell\}\cup\{\mu_{\ve_1}>\ell\}\cup\{\rho_{\ve_2}>\ell\}\cup\{\mu_{\ve_2}>\ell\}}\left(\rho_{\ve_1}+\rho_{\ve_2}\right)dxdt.\label{mt10}
		%\nonumber\\
		%&\leq&\int_{\{\rho_{\ve_1}\leq\ell\}\cap\{\mu_{\ve_1}\leq\ell\}\cap\{\rho_{\ve_2}\leq\ell\}\cap\{\mu_{\ve_2}\leq\ell\}}|a^2(\rho_{\ve_1}\mu_{\ve_1})b(\rho_{\ve_1})d(\rho_{\ve_1})-a^2(\rho_{\ve_2}\mu_{\ve_2})b(\rho_{\ve_2})d(\rho_{\ve_2})|dxdt\nonumber\\
		%&&+c
	\end{eqnarray}
		The first term in \eqref{pl88} implies that 
		\begin{equation}\label{mt23}
		\mbox{	$\{(\re,\me)\}$ is uniformly integrable.}
		\end{equation} Indeed, for each $E\subset\oT$ and $ s>1$, we have
	\begin{eqnarray}
		\int_E\re dxdt&=& \int_{E\cap\{\re\leq s\}}\re dxdt+\int_{E\cap\{\re> s\}}\re dxdt\nonumber\\
		&\leq& s|E|+\frac{1}{\ln s}\ioT\re\ln\re dxdt\leq  s|E|+\frac{c}{\ln s}.
	\end{eqnarray}
Thus, for each $\sigma>0$ we can find an $ s>1$ such that
$$\frac{c}{\ln s}\leq\frac{\sigma}{2}.$$
Whenever $|E|<\frac{\sigma}{2 s}$, we have
\begin{equation}\label{mt25}
\int_E\re dxdt\leq \sigma.	
\end{equation}
%$$$$
The same argument also works for $\{\me\}$. This completes the proof of \eqref{mt23}.

Since we have
\begin{equation}
	\left|\{\rho_{\ve_1}>\ell\}\cup\{\mu_{\ve_1}>\ell\}\cup\{\rho_{\ve_2}>\ell\}\cup\{\mu_{\ve_2}>\ell\}\right|\leq \frac{c}{\ell},
\end{equation}
for each $\sigma>0$ we can find an $\ell>0$ such that
	$$\int_{\{\rho_{\ve_1}>\ell\}\cup\{\mu_{\ve_1}>\ell\}\cup\{\rho_{\ve_2}>\ell\}\cup\{\mu_{\ve_2}>\ell\}}\re dxdt\leq \sigma\ \ \mbox{for all $\ve$.}$$
	Combining this with \eqref{mt10} gives
	\begin{equation}
		\mbox{$\{a(\re\me)\sqrt{\re}\}$ is precompact in $L^2(\oT)$.}
		%	\re\ra\rho,\ \ \me\ra\mu\ \ \mbox{weakly}
	\end{equation}
	Similarly, 
	\begin{equation}
		\mbox{$\{a(\re\me)\sqrt{\me}\}$ is precompact in $L^2(\oT)$.}
		%	\re\ra\rho,\ \ \me\ra\mu\ \ \mbox{weakly}
	\end{equation}
	The rest of the proof is similar to the one for \eqref{mt7}. We shall omit it here.
\end{proof}
In view of the Dunford-Pettis theorem and \eqref{mt7}, we may assume
\begin{eqnarray}
	\re&\ra&\rho,\ \ \me\ra\mu\ \ \mbox{weakly in $L^1(\oT)$},\label{mt24}\\
	\theta_1(\sqrt{\re\me})&\ra&\zeta \ \ \mbox{strongly in $L^2(\oT)$ and a.e. on $\oT$}.\label{mt22}
\end{eqnarray}
\begin{lem}\label{l34} We have
	\begin{eqnarray*}
		\zeta&=&\sqrt{\rho\mu}\ \ \mbox{on $\{\zeta<1\}$},\\
		\sre&\ra&\sr,\ \ \sme\ra\sm\ \ \mbox{strongly in $L^2(\{\zeta<1\})$}.
	\end{eqnarray*}
	%=(\xe+\ve)(1-\theta_1(\sqrt{\re\me}))\ra\xi(1-\zeta)
\end{lem}
\begin{proof} In view of  \eqref{mt24}, it is enough for us to show
	\begin{equation}\label{ms11}
\sqrt{\re(z)}\ra\sqrt{\rho(z)},\ \ \sqrt{\me(z)}\ra\sqrt{\mu(z)}	\ \ \mbox{for a.e. $z\in\{\zeta<1\}$}.
	\end{equation}
	First we conclude from \eqref{mt22} and Lemma \ref{l33} that
	\begin{equation}
\mbox{	$\theta_1(\sqrt{\re\me})$ and $	\sqrt{\re}(1-\theta_1(\sqrt{\re\me}))$ converge strongly in $L^2(\oT)$ and a.e. on $\oT$. }
	\end{equation}
	Let $z\in\{\zeta<1\}$ be given. Then for $\ve$ sufficiently small, we have $$\sqrt{\re(z)\me(z)}=\theta_1(\sqrt{\re(z)\me(z)}))<1.$$
	Consequently,
	\begin{eqnarray*}
	\sqrt{\re(z)}=\frac{\sqrt{\re(z)}(1-\theta_1(\sqrt{\re(z)\me(z)}))}{1-\theta_1(\sqrt{\re(z)\me(z)})}.	
	\end{eqnarray*}
	This implies the first limit in \eqref{ms11}. The second one can be established in an entirely similar manner.
	\end{proof}
An easy consequence of the preceding lemmas is that for $a\in W^{1,\infty}_0(-\infty,1),\ b, d\in C_{\textup{c}}(\mathbb{R})$ we have
\begin{eqnarray*}
	a(\sqrt{\re\me})&\ra& a(\sqrt{\rho\mu})\ \ \mbox{strongly in $L^p(\oT)$ for each $p\geq 1$}, \\
b(\sre)	a(\sqrt{\re\me})&\ra&b(\sr) a(\sqrt{\rho\mu})\ \ \mbox{strongly in $L^p(\oT)$ for each $p\geq 1$}, \\
b(\sre)	d(\sme)a(\sqrt{\re\me})&\ra&b(\sr)d(\sm) a(\sqrt{\rho\mu})\ \ \mbox{strongly in $L^p(\oT)$ for each $p\geq 1$.}
\end{eqnarray*}
Moreover,
\begin{equation}\label{happy2}
	\theta_1(\sqrt{\re\me})\ra \theta_1(\sqrt{\rho\mu})\ \ \mbox{strongly in $L^p(\oT)$ for each $p\geq 1$.}
\end{equation}
\begin{lem}We have
	\begin{eqnarray}
		\theta_{\frac{1}{\ve}}(\sre)\theta_1(\sqrt{\re\me})\nabla u_\ve&\ra& \mathbf{f}_1\ \ \mbox{weakly in $\left(L^1(\oT)\right)^N$,}\label{mt26}\\
		\theta_{\frac{1}{\ve}}(\sme)\theta_1(\sqrt{\re\me})\nabla u_\ve&\ra& \mathbf{f}_2\ \ \mbox{weakly in $\left(L^1(\oT)\right)^N$,}
	\end{eqnarray}
Moreover,
\begin{equation}\label{mt27}
	\mathbf{f}_1=\sqrt{\rho}\sqrt{\rho\mu}\nabla u,\ \ 	\mathbf{f}_2=\sqrt{\mu}\sqrt{\rho\mu}\nabla u\ \ \mbox{on $\{\sqrt{\rho\mu}<1\}$.}
\end{equation}	
\end{lem}
\begin{proof}According to the Dunford-Pettis theorem, all we need to do is to show that $\{\theta_{\frac{1}{\ve}}(\sre)\theta_1(\sqrt{\re\me})\nabla u_\ve\}$ is uniformly integrable. For each measurable set $E\subset\oT$ we derive from \eqref{pl88} that
	\begin{eqnarray*}
		\int_E|\theta_{\frac{1}{\ve}}(\sre)\theta_1(\sqrt{\re\me})\nabla u_\ve|dxdt\leq \left(\int_E\re dxdt\right)^{\frac{1}{2}}\left(\int_E|\nabla \ue|^2dxdt\right)^{\frac{1}{2}}\leq c\left(\int_E\re dxdt\right)^{\frac{1}{2}}.
	\end{eqnarray*}
The rest follows from the proof of \eqref{mt25}. This yields \eqref{mt26}. By \eqref{mt22} and Lemma \ref{l34},	
\begin{eqnarray*}
	\theta_{\frac{1}{\ve}}(\sre)\theta_1(\sqrt{\re\me})&\ra&\sqrt{\rho}\sqrt{\rho\mu}\ \ \mbox{strongly in $L^2(\{\sqrt{\rho\mu}<1\} )$},\\
\nabla\ue&\ra&\nabla u\ \ \mbox{weakly in $\left(L^2(\{\sqrt{\rho\mu}<1\} )\right)^N$. }
\end{eqnarray*}
This implies the first equation in \eqref{mt27}. The same proof also works for the second sequence. The proof is complete.	
\end{proof}
\begin{lem}We have
	\begin{eqnarray}
	\sre	\left((1+\ve)-\frac{\theta_{\frac{1}{\ve}}(\sqrt{\re})}{\sre}\theta_1(\sqrt{\re\me})\right)\nabla\sre&\ra&\mathbf{g}_1\ \ \mbox{weakly in $\left(L^1(\oT)\right)^N$,}\label{mt30}\\
		\sme	\left((1+\ve)-\frac{\theta_{\frac{1}{\ve}}(\sqrt{\me})}{\sme}\theta_1(\sqrt{\re\me})\right)\nabla\sme&\ra&\mathbf{g}_2\ \ \mbox{weakly in $\left(L^1(\oT)\right)^N$.}\label{mt31}
	\end{eqnarray}
Furthermore,
\begin{eqnarray}
	\mathbf{g}_1&=&\left\{\begin{array}{ll}
		\mathbf{0}& \mbox{on $\{\sqrt{\rho\mu}\geq1\}$ },\\
		\sr(1-\sqrt{\rho\mu})\nabla\sr& \mbox{on $\{\sqrt{\rho\mu}<1\}$ },
	\end{array}\right.\\
	\mathbf{g}_2&=&\left\{\begin{array}{ll}
	\mathbf{0}& \mbox{on $\{\sqrt{\rho\mu}\geq1\}$ },\\
	\sm(1-\sqrt{\rho\mu})\nabla\sm& \mbox{on $\{\sqrt{\rho\mu}<1\}$ },
\end{array}\right.
\end{eqnarray}
where $\nabla\sr, \ \nabla\sm$ are defined as in (D2).
	%The sequences $\left\{\left((1+\ve)\sre-\theta_{\frac{1}{\ve}}(\sqrt{\re})\theta_1(\sqrt{\re\me})\right)\nabla\sre\right\}$ and $ $ are bounded in $\left(L^2(\oT)\right)^N$.
	\end{lem}
\begin{proof}
	Note from \eqref{sdef2} that
	\begin{eqnarray}
		%	\left|	\left((1+\ve)\sre-\theta_{\frac{1}{\ve}}(\sqrt{\re})\theta_1(\sqrt{\re\me})\right)\nabla\sre\right|
	\lefteqn{\left((1+\ve)-\frac{\theta_{\frac{1}{\ve}}(\sqrt{\re})}{\sre}\theta_1(\sqrt{\re\me})\right)\left|\nabla\sre\right|}\nonumber\\
%	&=&		\nonumber\\
	&=&	\left((1+\ve)-S_\ve\theta_1(\sqrt{\re\me})+\frac{\theta_{\frac{1}{\ve}}(\sqrt{\me})}{\sme}\theta_1(\sqrt{\re\me})\right)\left|\nabla\sre\right|	\nonumber\\
	&\leq&\left(2(1+\ve)-S_\ve\theta_1(\sqrt{\re\me})\right)\left|\nabla\sre\right|
	\end{eqnarray}
By \eqref{pl88}, the sequence is bounded in  $\left(L^2(\oT)\right)^N$. Remember that $\{\re\}$ is uniformly integrable, and so is the sequence in \eqref{mt30},
% and the proof of \eqref{mt26}, the sequence $\left\{\left((1+\ve)\sre-\theta_{\frac{1}{\ve}}(\sqrt{\re})\theta_1(\sqrt{\re\me})\right)\nabla\sre\right\}$ is uniformly integrable, 
from which \eqref{mt30} follows. The same is true of the other sequence.

We can write
\begin{eqnarray}
	\lefteqn{\sre\left((1+\ve)-\frac{\theta_{\frac{1}{\ve}}(\sqrt{\re})}{\sre}\theta_1(\sqrt{\re\me})\right)\nabla\sre}\nonumber\\
	&=&\sqrt{\left((1+\ve)-\frac{\theta_{\frac{1}{\ve}}(\sqrt{\re})}{\sre}\theta_1(\sqrt{\re\me})\right)}\left(\sqrt{\re\left((1+\ve)-\frac{\theta_{\frac{1}{\ve}}(\sqrt{\re})}{\sre}\theta_1(\sqrt{\re\me})\right)}\ \ \nabla\sre\right) .\label{happy1}
\end{eqnarray}
It is not difficult for us to see that
$$\left((1+\ve)-\frac{\theta_{\frac{1}{\ve}}(\sqrt{\re})}{\sre}\theta_1(\sqrt{\re\me})\right)\ra 1-\theta_1(\sqrt{\rho\mu})\ \ \mbox{a.e on $\Omega_T$.}$$
Indeed, in view of \eqref{happy2}, it is enough for us to show
\begin{equation}\label{happy3}
	\frac{\theta_{\frac{1}{\ve}}(\sqrt{\re})}{\sre}\ra 1\ \ \mbox{a.e. on $\oT$.}
\end{equation}
% on the right-hand side of \eqref{happy1}
To this end, we estimate
$$\ioT\left(\frac{\sre}{\theta_{\frac{1}{\ve}}(\sqrt{\re})}-1\right)dxdt=\int_{\{\sre\geq\frac{1}{\ve}\}}\left(\ve\sre-1\right)dxdt\leq \ve\int_{\{\sre\geq\frac{1}{\ve}\}}\sre dxdt\leq c\ve.$$
Now the first factor in \eqref{happy1} converges weak$^*$ in $L^\infty(\oT)$, strongly in $L^p(\oT)$ for each $p>1$, and a.e. in $\oT$, while the second factor converges weakly in $\left(L^1(\oT)\right)^N$. By applying Egoroff's theorem (\cite{EG}, p.16) appropriately, we can still take $\ve\ra 0$ in the product and the limit has the factor $1-\theta_1(\sqrt{\rho\mu})$. Therefore, we must have
$$\mathbf{g}_1=\mathbf{0}\ \ \mbox{on $\{\sqrt{\rho\mu}\geq 1\}$}.$$
The same argument goes for $\mathbf{g}_2$.

For each $a\in W^{2,\infty}_0(-\infty,0),\ b\in C_{\textup{c}}^1(\mathbb{R})$ we easily see that
\begin{equation}
	\mbox{$a(\sqrt{\re\me}),\ a(\sqrt{\re\me}) b (\sre), a(\sqrt{\re\me}) b (\sme) $ are all bounded in $L^1(0, T; W^{1,1}(\Omega))$.}
\end{equation}
Indeed, there hold
\begin{eqnarray*}
	|a(\sqrt{\re\me})|&=&|a(\sqrt{\re\me})-a(1)|\leq \|a^\prime\|_{\infty, (-\infty,1)}(1-\theta_1(\sqrt{\re\me})),\\
	|a^\prime(\sqrt{\re\me})|&=&|a^\prime(\sqrt{\re\me})-a^\prime(1)|\leq \|a^{\prime\prime}\|_{\infty, (-\infty,1)}(1-\theta_1(\sqrt{\re\me})).
\end{eqnarray*}
Subsequently,
\begin{eqnarray*}
	|\nabla a(\sqrt{\re\me})|&\leq&c\sme(1-\theta_1(\sqrt{\re\me}))|\nabla\sre|+c\sre(1-\theta_1(\sqrt{\re\me}))|\nabla\sme|,\\
	\sre|\nabla\left( a(\sqrt{\re\me}) b (\sre)\right)|&\leq&c	|\nabla a(\sqrt{\re\me})|+ c(1-\theta_1(\sqrt{\re\me}))|\nabla\sre|,\nonumber\\
\sme|\nabla\left( a(\sqrt{\re\me}) b (\sme)\right)|	&\leq&c	|\nabla a(\sqrt{\re\me})|+ c(1-\theta_1(\sqrt{\re\me}))|\nabla\sme|.
\end{eqnarray*}
Note that all the terms on the right-hand sides are uniformly integrable. This together with \eqref{acom} and Lemmas \ref{l33} and \ref{l34} enables us to assume
\begin{eqnarray*}
	a(\sqrt{\re\me})&\ra &a(\sqrt{\rho\mu})\ \ \mbox{weakly in $L^1(0, T; W^{1,1}(\Omega))$},\\ \sre a(\sqrt{\re\me}) b (\sre)&\ra&\sr\ a(\sqrt{\rho\mu}) b (\sr)\ \ \mbox{weakly in $L^1(0, T; W^{1,1}(\Omega))$},\\
	\sme a(\sqrt{\re\me}) b (\sme)&\ra&\sm\ a(\sqrt{\rho\mu}) b (\sm)\ \ \mbox{weakly in $L^1(0, T; W^{1,1}(\Omega))$}.
\end{eqnarray*}
We easily see
\begin{eqnarray}
\lefteqn{	a(\sqrt{\re\me})b^\prime(\sre)	\sre\left((1+\ve)-\frac{\theta_{\frac{1}{\ve}}(\sqrt{\re})}{\sre}\theta_1(\sqrt{\re\me})\right)\nabla\sre}\nonumber\\
&=&\left((1+\ve)-\frac{\theta_{\frac{1}{\ve}}(\sqrt{\re})}{\sre}\theta_1(\sqrt{\re\me})\right)\sre\left(\nabla\left(a(\sqrt{\re\me}) b (\sre)\right)- b (\sre)\nabla a(\sqrt{\re\me})\right).\label{happy4}
\end{eqnarray}
The situation here is exactly the same as the one on the right-hand of \eqref{happy1}. Therefore, with a suitable application of Egoroff's theorem, we can pass to the limit in \eqref{happy4}
%The left-hand side can be viewed as a product of a strong convergent sequence and a weak convergent one in $L^2$. To justify  on the right-hand side, we must apply  (\cite{EG}, p.16) suitably. We shall omit the details. Upon taking $\ve\ra 0$ in \eqref{happy4}, we 
to arrive at 
\begin{eqnarray*}
	a(\sqrt{\rho\mu})b^\prime(\sr)\mathbf{g}_1&=&\sr(1-\theta_1(\sqrt{\rho\mu}))\left(\nabla\left(a(\sqrt{\rho\mu})b(\sr)-b(\sr)\nabla	a(\sqrt{\rho\mu})\right)\right).
\end{eqnarray*}
If $a, b$ are given as in (D2), then we have the desired  result. 
Condition \eqref{mt31} can be derived in a similar manner. The proof is complete.
%Obviously, the sequence $\{\left((1+\ve)-\frac{\theta_{\frac{1}{\ve}}(\sqrt{\re})}{\sre}\theta_1(\sqrt{\re\me})\right)\left|\nabla\sre\right|\}$ is bounded in 
\end{proof}

Finally, we can pass to the limit in system \eqref{pl55}-\eqref{pl66}. The proof of the main theorem is complete.
		
\end{document}